\documentclass[10pt,twocolumn,twoside]{IEEEtran}

\usepackage[cmex10]{amsmath}
\usepackage{amssymb,latexsym,amsfonts,stmaryrd}
\usepackage{amsthm}
\usepackage{graphicx}
\usepackage{tikz}
\usetikzlibrary{automata}
\usetikzlibrary{shapes}
\usepackage[americanresistors,americaninductors]{circuitikz}
\usepackage{steinmetz}
\usepackage{soul}
\usepackage{comment}
\usepackage{calrsfs}
\DeclareMathAlphabet{\pazocal}{OMS}{zplm}{m}{n}
\usepackage{algorithm}
\usepackage{algorithmic}
\usepackage{subfigure}
\usepackage{tikz-qtree}
\tikzset{edge from parent/.style={draw,very thick}}

\usepackage{tikz}
\usetikzlibrary{shapes,shadows,calc,snakes}
\usetikzlibrary{decorations}
\usepackage{pgfplots}

\newtheorem{theorem}{Theorem}
\newtheorem{lemma}{Lemma}
\newtheorem{proposition}{Proposition}
\newtheorem{definition}{Definition}
\newtheorem{example}{Example}
\newtheorem{problem}{Problem}

\newtheorem{assumption}{Assumption}

\renewcommand{\S}{{\mathbb{S}}}

\newcommand{\R}{{\mathbb{R}}}

\newcommand{\B}{{\mathbb B}}

\newcommand{\N}{{\mathbb{N}}}

\newcommand{\argmax}{\textrm{arg}\max}

\newcommand{\s}{\mathrm{span}}

\newcommand{\M}{\mathcal{M}}

\newcommand{\E}{\mathbb{E}}

\newcommand{\x}{{\mathbf{x}}}
\newcommand{\e}{{\mathbf{e}}}
\renewcommand{\O}{{{\mathcal{O}}}}
\newcommand{\J}{{{ \mathbf{J} }} }
\renewcommand{\M}{{{ \mathbf{M} }} }
\newcommand{\z}{{{\mathbf{z}}} }
\newcommand{\res}{{{ \mathbf{r} }} }

\renewcommand{\s}{{{ \mathbf{s}}} }

\renewcommand{\matrix}[1]{\begin{bmatrix}#1\end{bmatrix}}

\newcommand{\supp}{\textrm{supp}}

\begin{document}
\title{Secure State Estimation against Sensor Attacks in the Presence of Noise}
%
%
%

\author{\IEEEauthorblockN{Shaunak Mishra,
Yasser Shoukry,
Nikhil Karamchandani,
Suhas Diggavi and
Paulo Tabuada}
\thanks{
S. Mishra,
Y. Shoukry,
S. Diggavi and
P. Tabuada are with the
Electrical Engineering Department, University of California, Los Angeles,
CA 90095-1594, USA
(e-mail: \{shaunakmishra, yshoukry, suhasdiggavi, tabuada\}@ucla.edu).
N. Karamchandani
is with the
Electrical Engineering Department, Indian Institute of Technology Bombay,
Mumbai-400076, India (email: nikhilk@ee.iitb.ac.in).
The work was supported by NSF grant 1136174 and
DARPA under agreement number FA8750-12-2-0247.
A preliminary version of this work
appeared in the proceedings of ISIT 2015 \cite{ISIT15_shaunak}.
}
}

%
%

\markboth{}%
{}
%


\maketitle
\begin{abstract}
We consider the problem of estimating the state of a noisy linear dynamical system
when an unknown subset of sensors is arbitrarily corrupted by an adversary.
We propose a secure state estimation algorithm,
and derive (optimal) bounds on the achievable state estimation error given an upper bound on the number of attacked sensors. 
The proposed state estimator involves
Kalman filters operating over subsets of sensors to search for
a sensor subset which is reliable for state estimation.
To further improve the subset search time,
we propose Satisfiability Modulo Theory based techniques
to exploit the
combinatorial nature of searching over sensor subsets.
Finally, as a result of independent interest, we give a coding theoretic view of
attack detection and state estimation against sensor
attacks in a noiseless dynamical system.
 \end{abstract}
%
\IEEEpeerreviewmaketitle

\section{Introduction} \label{sec:introduction}
Securing cyber-physical systems (CPS) is a problem of growing importance as the vast majority of today's critical infrastructure is managed by such systems.
In this context, it is crucial
to understand the fundamental limits
for state estimation, an integral aspect of CPS, in the presence of malicious attacks.
With this motivation,
we focus on securely estimating the state of a linear dynamical system
from a set of noisy and maliciously corrupted sensor measurements.
We restrict the sensor attacks to be sparse in nature, \emph{i.e.}, an adversary can arbitrarily corrupt
an unknown subset of sensors in the system but is restricted by an upper bound
on the number of attacked sensors.

\par
Several recent works have studied
the problem of secure state estimation against sensor attacks in linear dynamical systems.
For setups with no noise in sensor measurements,
the results reported in \cite{Bullo_CSM, Hamza_TAC,YasserETPGarXiv} show that,
given a strong notion of observability,
(sparse) sensor attacks
can always be detected and isolated,
and we can exactly estimate the state of the system.
However,
with noisy sensors,
it is not trivial to
distinguish between the noise and the attacks injected by an adversary.
Prior work on state estimation with
sensor attacks in the presence of noise
can be broadly divided into two categories depending on the noise model:
1) bounded non-stochastic noise, and 2) Gaussian noise.
Results reported in \cite{Yasser_SMT,Joao_ACC,Pajic_ICCPS}
deal with bounded non-stochastic noise.
Though they provide sufficient conditions for distinguishing the sparse attack vector from bounded noise,
they
do not guarantee the optimality of their estimation algorithm.
The problem we focus on in this paper falls in the second category,
\emph{i.e.},
sensor attacks in the presence of Gaussian noise.
Prior work in this category
includes \cite{YilinAllerton,Bai_Gupta,BoydKF, Georgios_TSP2}.
In \cite{YilinAllerton}, the focus is on detecting a class of sensor attacks called \textit{replay} attacks
where the attacker replaces
legitimate sensor outputs
with outputs from previous time instants.
In \cite{Bai_Gupta},
the performance degradation of a scalar Kalman filter (\emph{i.e.}, scalar state and a single sensor) is studied when the (single) sensor is under attack.
They do not study attack sparsity across multiple sensors,
and in addition,
they focus on an adversary whose objective is to degrade the estimation performance without being detected (leading to a restricted class of sensor attacks).
In \cite{BoydKF} and \cite{Georgios_TSP2},
robustification approaches for state estimation against sparse sensor attacks are studied.
However, they lack optimality guarantees against arbitrary sensor attacks.
\par In this paper,
we study a general linear dynamical system
with process and sensor noises having
a Gaussian distribution,
and
give (optimal) guarantees on the achievable state estimation error against arbitrary sensor attacks.
The following toy example is illustrative of the nature of the problem addressed in this paper and some of the ideas behind our solution.

\begin{example}
Consider a linear dynamical system with a scalar state $x(t)$ such that $x(t+1) = x(t) + w(t)$, 
and three sensors (indexed by $d \in \{1,2,3\}$) with outputs $y_d(t) = x(t) + v_d(t)$;
where $w(t)$ and $v_d(t)$ are the process noise and sensor noise at
sensor $d$ respectively. The process and sensor noises follow
a zero mean Gaussian distribution with i.i.d. instantiations over time.
The sensor noise is also independent across sensors.
Now, consider an adversary which can attack any one of the sensors in the system and arbitrarily change its output.
In the absence of sensor noise, it is trivial to detect such an 
attack since the two good sensors (not attacked by the adversary) will have the same output.
Hence, a majority based rule on the outputs leads to the exact state.
However, in the presence of sensor noise,
a difference in outputs across sensors can also be attributed to the noise,
and thus cannot be considered an attack indicator.
As a consequence of results in this paper,
in this example we can identify a subset of two sensors which can be reliably
used for state estimation despite an adversary who can attack any one of the three noisy sensors.
In particular, our approach for this example would be to
search for a subset of two sensors which satisfy the following check: over a large enough time window,
the outputs from the two sensors are \textit{consistent} with the Kalman state estimate based on outputs from the same subset of sensors.
Furthermore, we can show that such an approach leads to the optimal state estimation error for the given adversarial setup.
\end{example}

\par
In this paper, we generalize the Kalman filter based approach in the above example to a general linear dynamical system with sensor and process noise.
The Kalman estimate based check mentioned in the above example forms the basis of a detector for an \textit{effective} attack; a notion that we introduce in this paper.
For state estimation,
we search for a sensor subset which passes such an
effective attack detector,
and then use outputs from such a sensor subset for state estimation.
We also derive impossibility results (lower bounds)
on the state estimation error
in our adversarial setup, and show that our
proposed state estimation algorithm is optimal
in the sense that it achieves these lower bounds.
To further reduce the sensor subset search time for the state estimator,
we propose Satisfiability Modulo Theory (SMT)
based techniques to harness the combinatorial nature of the search problem,
and demonstrate the improvements in search time through numerical experiments.

As a result of independent interest,
we give a coding theoretic interpretation (alternate proof) for the necessary and sufficient conditions for secure state estimation in the absence of noise \cite{ Hamza_TAC, YasserETPGarXiv, Joao_ACC} (known as the sparse observability condition). In particular,
we relate the sparse observability condition required for attack detection and
secure state estimation in dynamical systems
to the Hamming distance requirements for error detection and correction \cite{blahut} in
classical coding theory.

The remainder
of this paper\footnote{Compared to the preliminary version \cite{ISIT15_shaunak},
this paper differs in the presentation
of results through effective attack detection.
In addition, we reduce the complexity of the state estimation algorithm in
\cite{ISIT15_shaunak} and also describe SMT based techniques for reducing the subset
search time.}
is organized as follows.
Section~\ref{sec:setup} deals with the setup and problem formulation.
In Section~\ref{sec:detector}, we describe our effective attack detector followed by Section~\ref{sec:main_results} on our main results for effective
attack detection and secure state estimation.
Section~\ref{sec:smt} deals with SMT based techniques and 
Section~\ref{sec:exp} with the experimental results.
Finally, Section~\ref{sec:sparse_proof} describes the coding theoretic view for attack detection and secure state estimation.

\section{Setup} \label{sec:setup}
In this section,
we discuss the adversarial setup along with assumptions on the underlying dynamical system,
and provide a mathematical formulation of the state estimation problem considered in this paper.
\subsection{Notation}
The symbols $\N, \R$ and $\B$ denote the sets of natural, real, and Boolean numbers respectively. The symbol $\land$ denotes the logical AND operator.
The support of a vector $ \mathbf{x} \in \R^n$,  denoted by $\supp(\mathbf{x})$,
is the set of indices of the non-zero elements of $\mathbf{x}$.
If $\s$ is a set, $|\s|$ is the cardinality of $\s$. 
For the matrix $\mathbf{M} \in \R^{m \times n}$, 
unless stated otherwise,
we denote by $\mathbf{M}_i \in \R^{1\times n}$ the $i${th} row of the matrix.
For the set \mbox{$\s \subseteq \{1, \hdots, m\}$}, we denote by
$\mathbf{M}_{\s} \in \R^{\vert \s \vert \times n}$
the matrix obtained from $\mathbf{M}$ by  removing all the rows except those
indexed by $\s$.
We use $tr \left( \mathbf{M} \right)$ to denote the trace of the matrix $\mathbf{M}$.
If the matrix $\mathbf{M}$ is symmetric,
we use $\lambda_{\min} \left( \mathbf{M} \right)$ and
$\lambda_{\max} \left( \mathbf{M} \right)$
to denote the minimum and maximum eigenvalue of $\mathbf{M}$ respectively.
We denote by $\S_{+}^{n}$ the set of all $n\times n$ positive semi-definite matrices.
For a random variable $\x \in \R^n$,
we denote its mean by $\E \left ( \x \right) \in \R$
and its covariance by $Var(\x)  \in \S_{+}^{n} $.
For a discrete time random process $\{ \x (t) \}_{t \in \N }$,
the sample average of $\x$
using $N$ samples starting at time $t_1$ is defined as follows:
\begin{align}
\E_{N,t_1} \left ( \x \right)  &= \frac{1}{N} \sum_{t = t_1}^{t_1+N-1} \x{(t)}. \label{eq:sample_avg_defn}
\end{align}
We denote by $\mathbf{I}_m \in \R^{m\times m}$ and $\mathbf{1}_m \in \R^{m \times 1}$ the identity matrix of dimension $m$ and the vector of all ones respectively.
The notation $\x(t) \sim \mathcal{N} \left(\boldsymbol{\mu}, \boldsymbol{\Omega} \right)$
is used to denote an i.i.d. Gaussian random process with mean $\boldsymbol{\mu}$ and covariance matrix $\boldsymbol{\Omega}$.
Finally, we use the symbol $\preccurlyeq$ for element-wise comparison between matrices. That is, for two matrices $\mathbf{A}$ and $\mathbf{B}$ of the same size,
$\mathbf{A} \preccurlyeq \mathbf{B}$
is true if and only if each element $a_{i,j}$ is smaller than or equal to $b_{i,j}$.

\subsection{System model} 
We consider a linear dynamical system $\pmb{\Sigma}_a $ with sensor attacks as shown below:
\begin{align}
\pmb{\Sigma}_a 
\begin{cases}
\x \left(t+1\right) &= \mathbf{A} \x(t) + \mathbf{B} \mathbf{u}(t) + \mathbf{w}(t), \\ 
\mathbf{y}(t) &= \mathbf{C} \x(t) + \mathbf{v}(t) + \mathbf{a}(t) ,
\end{cases}
\label{eq:vector_system_model}
\end{align} 
where $\mathbf{x}(t)\in \mathbb{R}^n$ denotes the state of the plant at time $t \in \mathbb{N}$,
$\mathbf{u}(t) \in \mathbb{R}^m$ denotes the input at time $t$,
$\mathbf{w}(t) \sim \mathcal{N} \left ( \mathbf{0}, \sigma^2_w \mathbf{I}_n\right)$ denotes the process noise at time $t$,
$\mathbf{y}(t) \in \mathbb{R}^p$ denotes the output of the plant at time $t$ and $ \mathbf{v}(t) \sim \mathcal{N} \left ( \mathbf{0}, \sigma^2_v \mathbf{I}_p\right)$ denotes the sensor noise at time $t$. Both $\mathbf{v}(t)$ and $\mathbf{w}(t)$
have i.i.d. instantiations
over time, and $\mathbf{v}(t)$ is independent of $\mathbf{w}(t)$.
In addition,
we denote the output and (sensor) noise at sensor $i \in \{1,2,\ldots, p\}$
at time $t$
as
$y_i(t) \in \R$ and $v_i(t) \in \R$ respectively.
We assume that the input $\mathbf{u}(t)$ is known at all time.
Hence, its contribution to the output $\mathbf{y}(t)$ is also known, and therefore, $\mathbf{u}(t)$ can be ignored.
That is, for the rest of the paper, and without loss of generality,
we consider the case of $\mathbf{u}(t) = 0$ for all time $t \in \N$.

\par The sensor attack vector $\mathbf{a}(t) \in \R^p$ in \eqref{eq:vector_system_model} is introduced by a $k$-adversary defined as follows.
\begin{assumption}
A $k$-adversary can corrupt any $k$ out of the $p$ sensors in the system.
\end{assumption}
Specifically, let $\pmb{\kappa} \subseteq \{1,2,\ldots, p\}$ denote the set of attacked sensors
(with $|\pmb{\kappa}| = k$). The $k$-adversary can observe the actual outputs in the $k$ attacked sensors and change them arbitrarily. For an attack free sensor $j \notin \pmb{\kappa}$,
$\mathbf{a}_j(t) =0$ for all time $t \in \N$.
\begin{assumption}
The adversary's choice of $\pmb{\kappa}$ is unknown but is assumed to be constant over time (static adversary).
\end{assumption}
\begin{assumption}
The adversary is assumed to have unbounded computational power, and knows the system parameters (\emph{e.g.,} $\mathbf{A}$ and $\mathbf{C}$) and noise statistics (\emph{e.g.,} $\sigma^2_w$ and $\sigma^2_v$). 
\end{assumption}
However, the adversary is limited to have only causal knowledge of the process and
sensor noise as stated by the following two assumptions.
\begin{assumption}
The adversary's knowledge at time $t$ is statistically independent of $\mathbf{w}(t')$ for $t' > t$, i.e., $\mathbf{a}(t)$ is statistically independent of $\{\mathbf{w}(t')\}_{t'>t}$.
\label{ass:A1}
\end{assumption}
\begin{assumption}
For an attack-free sensor \mbox{$i \in \{1,2,\ldots, p\} \setminus \pmb{\kappa}$},
the adversary's knowledge at time $t$ (and hence $\mathbf{a}(t)$) is statistically independent of $\{ v_i(t') \}_{t' > t}$.
\label{ass:A2}
\end{assumption}
Intuitively, Assumptions~\ref{ass:A1} and \ref{ass:A2} limit the adversary to have only causal knowledge of the process noise and the sensor noise in \textit{good} sensors (not attacked by the adversary).
Note that, apart from Assumptions~\ref{ass:A1} and \ref{ass:A2}, we do not impose any restrictions on the statistical properties, boundedness and the time evolution of the corruptions introduced by the $k$-adversary.

In the following subsections,
we first introduce the (effective) attack detection problem, followed by
the (optimal) secure state estimation problem.
As we show later in the paper (in Section~\ref{sec:main_results}),
our solution for the effective attack detection problem
is used as a crucial component for solving the
secure state estimation problem.

\subsection{Effective Attack Detection Problem}
In this section,
we introduce our notion of effective (sensor) attacks
and formulate the problem of detecting them.
Recall that in the absence of sensor attacks,
using a Kalman filter for estimating the state in \eqref{eq:vector_system_model}
leads to the (optimal) minimum mean square error (MMSE) covariance asymptotically \cite{kailath_book}.
In this context,
our notion of effective attacks is based on the following intuition:
if we naively use a Kalman filter for state estimation in the presence of an adversary,
an attack is effective when it causes a higher empirical error variance compared to the attack-free case.
Before we formally state our definition of effective attacks,
we first setup some notation for Kalman filters as described below.

We denote by $\hat{\x}_{\mathbf{s}}(t)$
the state estimate of a Kalman filter at time $t$
using outputs till time $t-1$
from the sensor subset $\mathbf{s} \subseteq \{1,2,\ldots, p\}$.
Since we use outputs till time $t-1$,
we essentially use the \textit{prediction} version of a Kalman
filter as opposed to \textit{filtering} where outputs till time $t$ are used to compute $\hat{\x}_{\s}(t)$.
In this paper,
we state our results using the prediction version of the Kalman filter;
the extension for the filtering version is straightforward
(for details about the filtering version of our results, see Appendix~\ref{sec:filtering}).
In addition to $\hat{\x}_{\mathbf{s}}(t)$,
we denote by $\hat{\x}^{\star}_{\mathbf{s}}(t)$
the Kalman filter state estimate at time $t$ using sensor subset $\mathbf{s}$
when all the sensors in $\mathbf{s}$ are attack-free.
We eliminate the subscript $\mathbf{s}$ from the previous notation whenever the Kalman filter uses all sensor measurements, \emph{i.e.}, when $\mathbf{s} = \{1,\ldots,p\}$. 
In this paper, for the sake of simplicity,
we assume that all the Kalman filters we consider (in our proposed algorithms and their analysis)
are in steady state \cite{kailath_book}
when they use uncorrupted sensor outputs.
Hence, in the absence of attacks,
the error covariance matrix $\mathbf{P}^{\star} (t) \in \S_n^{+}$
defined as:
 $$\mathbf{P}^{\star} (t) = \mathbf{P}^{\star} = \E{ \left (  \left (
\x(t) - \hat{\x}^{\star}(t) \right ) \left ( \x(t) - \hat{\x}^{\star}(t) \right)^T \right)},
$$ does not depend on time.
In a similar spirit,
we define the error covariance matrix $\mathbf{P}_{\mathbf{s}}^{\star} \in \S^{+}_n$
corresponding to sensor subset $\mathbf{s} \subseteq \{1,2,\ldots, p\}$ as:
$$\mathbf{P}_{\mathbf{s}}^{\star} = \E{(\x(t) - \hat{\x}_\mathbf{s}^{\star}(t))(\x(t) - \hat{\x}_\mathbf{s}^{\star}(t))^T}.$$
Note that the
error covariance matrix depends on
the set of sensors involved in estimating the state.
Also, the steady state error has zero mean, \emph{i.e.},
$\E \left ( \x(t) - \hat{\x}_\mathbf{s}^{\star}(t) \right)  = 0 $.
Using the above notation, we define an $(\epsilon, \mathbf{s})$-effective attack  as follows.

\begin{definition}[\textbf{$(\epsilon,\mathbf{s})$-Effective Attack}]
Consider the linear dynamical system under attack
$\pmb{\Sigma_a}$
as defined in~\eqref{eq:vector_system_model},
and a $k$-adversary satisfying Assumptions 1-5.
For the set of sensors $\mathbf{s}$, an $\epsilon > 0$,
and a large enough $N \in \mathbb{N}$,
an attack signal is called $(\epsilon, \mathbf{s})$-effective
at
time $t_1$
if the following bound holds:
$$tr \left(\E_{N,t_1}{\left(\e_{\mathbf{s}} \e_{\mathbf{s}}^T\right)} \right) > tr( \mathbf{P}^{\star}_{\mathbf{s}})  + \epsilon,$$
where $\e_{\mathbf{s}}(t) = \x(t) - \hat{\x}_{\mathbf{s}}(t)$,
and $\E_{N,t_1}(\cdot)$ denotes the sample average as defined
\eqref{eq:sample_avg_defn}.
\label{def:effective_attack}
\end{definition}
In other words,
an attack is called $(\epsilon, \mathbf{s})$-effective
if it can lead to a higher estimation error compared to
the optimal estimation error in the absence of sensor attacks,
using the same set of sensors $\mathbf{s}$.
An attack is called $(\epsilon, \mathbf{s})$-ineffective if it is not $(\epsilon, \mathbf{s})$-effective.
Essentially, we use
$\E_{N,t_1}{\left(\e_{\mathbf{s}} \e_{\mathbf{s}}^T\right)}$ as
a \textit{proxy} for the state estimation error covariance matrix in the presence of attacks;
a sample average is used instead of an expectation because the resultant error
in the presence of attacks may not be ergodic.
Also, since $\hat{\x}_{\mathbf{s}}(t)$ is computed using
all measurements from time $0$ till time $t -1$,
Definition~\ref{def:effective_attack} implicitly takes into consideration the effect of attack signal $\mathbf{a}(t)$ for the
time window starting from $0$ till time $t + N - 1$.

Using the above notion of an $(\epsilon,\mathbf{s})$-effective attack, we define the $\epsilon$-effective attack detection problem as follows.
\begin{problem}{[\textbf{$\epsilon$-Effective Attack Detection Problem}]}
Consider the linear dynamical system under attack $\pmb{\Sigma_a}$
as defined in~\eqref{eq:vector_system_model},
and a $k$-adversary satisfying Assumptions 1-5.
Let $\mathbf{s}_{\text{all}}$ be the set of all sensors,
i.e., $\mathbf{s}_{\text{all}} = \{1,\ldots,p\}$.
Given an $\epsilon > 0$,
construct an attack indicator $\hat{d}_{\text{attack}} \in \{0,1\}$
such that:
\begin{align*}
\hat{d}_{\text{attack}}(t_1) = \begin{cases}
1 & \textbf{if } \; \text{the attack is $(\epsilon, \mathbf{s}_{\text{all}})$-effective at time $t_1$}\\
0 & \textbf{otherwise}.
\end{cases}
\end{align*}
\label{prob:effective_attack}
\end{problem}
\subsection{Optimal Secure State Estimation Problem}
We now focus on the problem of estimating
the state from the adversarially corrupted sensors.
We start by showing a negative result
stating that a certain estimation error bound may be impossible to achieve in the
presence of a $k$-adversary.
To do so, we define the sensor set
that contains $p - k$ sensors and corresponds to the worst case Kalman estimate as:
\begin{align}
 \mathbf{s}_{\text{worst},p - k} = \argmax_{ \substack{\mathbf{s} \subseteq \{1,2,\ldots , p\}, \\ |\mathbf{s}| = p-k} } tr (\mathbf{P}_{\mathbf{s}}^{\star} ).
 \end{align}
The impossibility result can now be stated as follows.
\begin{theorem}[\textbf{Impossibility}] \label{thm:impos}
Consider the linear dynamical system under attack
$\pmb{\Sigma_a}$
as defined in~\eqref{eq:vector_system_model},
and an oracle MMSE estimator that has knowledge of $\pmb{\kappa}$,
i.e., the set of sensors attacked by a $k$-adversary.
Then, there exists a choice of sensors $\pmb{\kappa}$ and an attack sequence $\mathbf{a}(t)$ such that the trace of the error covariance of the oracle estimator is bounded from below as follows:
\begin{align}
 tr \bigg ( \E \left ( \mathbf{e}(t) \mathbf{e}^T(t) \right )  \bigg)  \ge  tr \bigg(  \mathbf{P}^\star_{\mathbf{s}_{\text{worst},p - k}}  \bigg),
 \label{eq:lower_impossible_bound}
\end{align}
where $\mathbf{e}(t)$ above is the oracle estimator's error.
\end{theorem}

\begin{proof}
Consider the attack scenario where the outputs from all attacked sensors are equal to zero, \emph{i.e.}, the
corruption $\mathbf{a}_j(t) = - \mathbf{C}_j \mathbf{x}(t) - v_j(t), \; \forall j \in \pmb{\kappa}$.
Hence, the information collected from the attacked sensors cannot enhance the estimation performance. Accordingly, the estimation performance from the remaining sensors is the best 
one can expect to achieve. Hence, the result follows by picking $\pmb{\kappa}$ such that $\pmb{\kappa} = \{1,\ldots,p\}\setminus \mathbf{s}_{\text{worst},p - k}$.
\end{proof}
In the context of Theorem~\ref{thm:impos},
we define a state estimate to be optimal if it is guaranteed to achieve the lower bound shown in~\eqref{eq:lower_impossible_bound}.
This can be formalized as follows.

\begin{problem}{[\textbf{Optimal Secure State Estimation Problem}]}
Consider the linear dynamical system under attack $\pmb{\Sigma_a}$
as defined in~\eqref{eq:vector_system_model}, and 
a $k$-adversary satisfying Assumptions 1-5.
For a time window $G = \{ t_1, t_1 + 1 , \ldots , t_1 + N -1\}$,
construct the state estimates $\{\hat{\x}(t)\}_{t \in G }$ such that:
$$ 
  tr \bigg(\E_{N,t_1} \left( \mathbf{e}\mathbf{e}^T \right) \bigg)
 \le 
 tr \bigg(  \mathbf{P}^\star_{\mathbf{s}_{\text{worst},p - k}} \bigg),  $$
\label{prob:optimal_sse}
where $\mathbf{e}(t) = \x (t) - \hat{\x}(t)$ is the state estimation error.
\end{problem}
Similarly to Definition~\ref{def:effective_attack},
we use the sample average $\E_{N,t_1} \left( \mathbf{e}\mathbf{e}^T \right)$
in Problem~\ref{prob:optimal_sse}
(and not expectation)
since the resultant error in the presence of attacks may not be ergodic.

\section{ Sparse observability and $(\epsilon,\mathbf{s})$-effective attack detection}
\label{sec:detector}
In this section, we first describe the notion of $k$-sparse observability~\cite{YasserETPGarXiv}.
This notion plays an important role in determining when Problems~\ref{prob:effective_attack}
and~\ref{prob:optimal_sse} are solvable.
After describing sparse observability,
we describe an algorithm for $(\epsilon,\mathbf{s})$-effective attack detection
which leverages sparse observability for its performance guarantees.
\subsection{$k$-Sparse Observability}
\begin{definition}{\textbf{($k$-Sparse Observable System)}}
The linear dynamical system under attack $\pmb{\Sigma_a}$
as defined in~\eqref{eq:vector_system_model},
is said to be $k$-sparse observable
if for every set $\mathbf{s} \subseteq\{1,\hdots,p\}$ with $|\mathbf{s}| = p - k$, the pair
$(A,C_{{\mathbf{s}}})$ is observable.
\end{definition}
In other words, a system is $k$-sparse observable if it remains 
observable after eliminating any choice of $k$ sensors. 
In the absence of sensor and process noise,
the conditions under which exact (\emph{i.e.}, zero error)
state estimation can be done despite sensor attacks have been studied in \cite{Hamza_TAC,YasserETPGarXiv,Joao_ACC} where it is shown that $2{k}$-sparse observability is necessary and sufficient for exact state estimation against a $k$-adversary.
In Section~\ref{sec:sparse_proof}, we provide a coding theoretic interpretation for this condition in the context of attack detection and secure state estimation in
any noiseless dynamical system.

\subsection{$(\epsilon, \mathbf{s})$-Effective Attack Detector}
In this section, we describe an algorithm
based on the sparse observability condition
for
detecting an $(\epsilon, \mathbf{s})$-effective attack.
We first introduce some additional notation, followed by the description
of the algorithm and its performance guarantee.
\subsubsection{Additional notation}
Let the sensors be indexed by $i \in \{ 1,2, \ldots, p\}$. 
We define the following observability matrices:
\begin{align}
\O_{i} &=\begin{bmatrix} \mathbf{C}^T_{i}  \\ \mathbf{C}^T_{i} \mathbf{A}  \\ \vdots \\ \mathbf{C}^T_{i} \mathbf{A}^{\mu_i -1} \end{bmatrix}, 
\quad \O = \begin{bmatrix} \O_{1} \\ \O_{2}  \\ \vdots \\ \O_{p}  \end{bmatrix},
\label{eq:obs_defn}
\end{align}
where $\O_{i}$
is the observability matrix for sensor $i$ (with observability index $\mu_i$
as shown in \eqref{eq:obs_defn})
and $\O$ is the observability matrix for the entire system (\emph{i.e.}, $p$ sensors) formed by
stacking the observability matrices for the sensors.
Similarly, for any sensor subset $\mathbf{s} \subseteq \{1,2, \ldots, p\}$, we denote
the observability matrix for $\mathbf{s}$ by $\O_{\mathbf{s}}$ (formed by stacking
the observability matrices of sensors in $\mathbf{s}$).
Without loss of generality, we will consider the observability index $\mu_i =n$ for each sensor.
For any sensor subset $\mathbf{s}$ with $|\mathbf{s}| > k$, we define $\lambda_{\min,\mathbf{s} \setminus k}$ as follows:
\begin{align}
\lambda_{\min,\mathbf{s} \setminus k} = \min_{\mathbf{s}_1 \subset \mathbf{s}, \;
 |\mathbf{s}_1| =|\mathbf{s}|-k} 
\lambda_{\min} \left ( \O_{\mathbf{s}_1}^T \O_{\mathbf{s}_1}  \right), 
\end{align} 
where $\lambda_{\min}\left ( \O_{\mathbf{s}_1}^T \O_{\mathbf{s}_1}  \right)$ denotes
the minimum eigenvalue of $\O_{\mathbf{s}_1}^T \O_{\mathbf{s}_1}$.
We define matrices $\mathbf{J}_i$,
$\J$
and
$\M$ as shown below:
\begin{align}
\mathbf{J}_i & = \begin{bmatrix} \mathbf{0} & \mathbf{0}  & \hdots & \mathbf{0} \\
\mathbf{C}^T_i & \mathbf{0}  & \hdots & \mathbf{0} \\
\mathbf{C}^T_i \mathbf{A} & \mathbf{C}^T_i  & \hdots & \mathbf{0} \\
\vdots & \vdots & \ddots & \vdots \\
\mathbf{C}^T_i \mathbf{A}^{{\mu_i}-2} & \mathbf{C}^T_i \mathbf{A}^{{\mu_i}-3 } & \hdots & \mathbf{C}^T_i \\
\end{bmatrix}, \;\;
\J  = \begin{bmatrix} \mathbf{J}_1 \\  \mathbf{J}_2 
\\ \vdots \\ \mathbf{J}_p
 \end{bmatrix}, 
\nonumber \\
 \M & =  \sigma^2_w \J  \J^T   + \sigma^2_v \mathbf{I}_{n p}.
\end{align}
In a similar spirit,
$\J_{\mathbf{s}}$ is defined for a sensor subset $\mathbf{s}$
by stacking $\mathbf{J}_i$ for $i \in \s$, and
$\M_{\mathbf{s}} =  \sigma^2_w \J_{\s}  \J_{\s}^T   + \sigma^2_v \mathbf{I}_{n | s | } $.
We use the following notation for sensor outputs and noises corresponding to a time window of size
$\mu_i =n$ (observability index):
\begin{align}
\mathbf{y}_{i}(t) &= \begin{bmatrix}
y_i (t) \\ y_i (t+1) \\ \vdots \\  y_i  (t+ \mu_i -1)
\end{bmatrix}, \;
\mathbf{v}_{i}(t)  = \begin{bmatrix}
v_i (t) \\ v_i (t+1) \\ \vdots \\  v_i  (t+ \mu_i -1)
\end{bmatrix}, \nonumber \\
\bar{\mathbf{y}}(t) & = \begin{bmatrix}
\mathbf{y}_1(t) \\ \mathbf{y}_{2}(t)  \\ \vdots \\  \mathbf{y}_{p}(t)
\end{bmatrix}, \;
\bar{\mathbf{v}}(t) = \begin{bmatrix}
\mathbf{v}_1(t) \\ \mathbf{v}_{2}(t)  \\ \vdots \\  \mathbf{v}_{p}(t)
\end{bmatrix}, \;
\bar{\mathbf{w}}(t) = \begin{bmatrix}
\mathbf{w}(t) \\ \mathbf{w}(t+1)  \\ \vdots \\  \mathbf{w}(t+n-1)
\end{bmatrix},
\end{align}
where $y_i(t)$ and $v_i(t)$
denote the output and sensor noise at sensor $i$
at time $t$ respectively.
\subsubsection{Attack Detection Algorithm}
We consider the attack detection problem for a time window
$G = \{t_1, t_1 + 1, \ldots, t_1 + N-1 \}$
, and assume without loss of generality that
the window size $N$ is divisible by $n$.
\begin{algorithm}[t]
\caption{\textsc{Attack-Detect}$(\mathbf{s},t_1)$}
\begin{algorithmic}[1]
\STATE  Run a Kalman filter that uses all measurements
from sensors indexed by $\mathbf{s}$
until time $t_1 -1 $ and compute the estimate $\hat{\x}_{\mathbf{s}} (t_1) \in \R^n$.
\STATE Recursively repeat the previous step $N-1$ times to calculate all estimates $\hat{\x}_{\mathbf{s}} (t) \in \R^n$, $ \forall t  \in G = \{t_1, t_1 + 1, \ldots, t_1 + N-1 \}$.
\STATE For time $t \in G$,
calculate the \textit{block} residue:
\begin{align*}
\mathbf{r}_{\mathbf{s}}(t) =  \bar{\mathbf{y}}_{\mathbf{s}}(t) -  \O_{\mathbf{s}} \hat{\x}_{\mathbf{s}}(t)\quad  \forall t \in G.
\end{align*}
\IF{block residue test defined below holds,
\begin{align*}
 \E_{N,t_1} \left(  \res_{\mathbf{s}}  \res_{\mathbf{s}}^T \right)  - 
\left ( \O_{\mathbf{s}}  \mathbf{P}^\star_{\mathbf{s}} \O_{\mathbf{s}}^T   +
 \M_{\mathbf{s}} \right)
  \preccurlyeq   \eta \; \mathbf{1}_{n|{\mathbf{s}}|} \mathbf{1}^T_{n|{\mathbf{s}}|},
\end{align*}
where $0 < \eta \leq     \left( \frac {\lambda_{\min,\mathbf{s} \setminus k}} {3n(|{\mathbf{s}}|-k)} \right) \epsilon$ ,
}
\STATE assert $\hat{d}_{\text{attack},\mathbf{s}}(t_1) := 0$
\ELSE
\STATE assert $\hat{d}_{\text{attack},\mathbf{s}}(t_1):= 1$
\ENDIF
\STATE \textbf{return} $( \hat{d}_{\text{attack},\mathbf{s}}(t_1), \{ \hat{\x}_{\mathbf{s}}(t) \}_{t \in G} )$
\end{algorithmic}
\label{alg:detector}
\end{algorithm}
For a sensor subset $\mathbf{s}$ with $|\mathbf{s}| > k$, we start by computing the state estimate $\hat{\x}_{\mathbf{s}}(t_1)$ obtained through a Kalman filter that uses measurements  collected from time $0$ up to time $t_1 - 1$ from all sensors indexed by the subset $\mathbf{s}$.
Using this estimate,
we can calculate the \textit{block} residue
$\res_{\mathbf{s}}(t_1)$ which is the discrepancy
between the estimated output
$\hat{\overline{\mathbf{y}}}_{\mathbf{s}}(t_1) = \O_{\mathbf{s}} \hat{\x}_{\mathbf{s}}(t_1)$ and the actual output $\overline{\mathbf{y}}_{\mathbf{s}}(t_1)$, \emph{i.e.},
\begin{align}
 \res_{\mathbf{s}}(t_1) = \overline{\mathbf{y}}_{\mathbf{s}}(t_1) - \hat{\overline{\mathbf{y}}}_{\mathbf{s}}(t_1) =  \overline{\mathbf{y}}_{\mathbf{s}}(t_1) - \O_{\mathbf{s}} \hat{\x}_{\mathbf{s}}(t_1).
 \label{eq:fast_residue_test}
 \end{align}
By repeating the previous procedure $N-1$ times, we can obtain the sequence of residues $\{\res_{\s}(t)\}_{t \in G }$.
The next step is to calculate the
sample average of $\res_{\s}(t) \res_{\s}^T(t)$, and
compare the sample average with the expected value of $\res_{\s}(t) \res_{\s}^T(t)$ in the case when sensor subset $\mathbf{s}$ is attack-free.
This can be done using the following (block) residue test:
\begin{align}
 \E_{N,t_1} \left(  \res_{\mathbf{s}}  \res_{\mathbf{s}}^T \right)  - 
\left ( \O_{\mathbf{s}}  \mathbf{P}^\star_{\mathbf{s}} \O_{\mathbf{s}}^T   +
 \M_{\mathbf{s}} \right)
  \preccurlyeq   \eta \; \mathbf{1}_{n|{\mathbf{s}}|} \mathbf{1}^T_{n|{\mathbf{s}}|},
  \label{eq:fast_block_residue_test_prediction}
  \end{align}
for some $\eta > 0$.
Simply put, the residue test just checks whether the
sample average of $\res_{\mathbf{s}}(t)  \res_{\mathbf{s}}^T(t)$ over the time window $G$
is \textit{close} to its attack-free expected value $\O_{\mathbf{s}}  \mathbf{P}^\star_{\mathbf{s}} \O_{\mathbf{s}}^T   +
 \M_{\mathbf{s}}$ (details in Section~\ref{sec:detector_guarantees}).
It is crucial to recall that the attack-free estimation error covariance matrix
$\mathbf{P}^\star_{\mathbf{s}}$ used in \eqref{eq:fast_block_residue_test_prediction}
can be computed offline \cite{kailath_book}
without the need for any data collected from attack-free sensors.
If the element-wise comparison in the residue test
\eqref{eq:fast_block_residue_test_prediction} is valid,
we set the attack detection flag $\hat{d}_{\text{attack},\mathbf{s}}(t_1)$ to zero indicating that no attack was detected in sensor subset $\s$.
This procedure is summarized in Algorithm~\ref{alg:detector}. 
\subsection{Performance Guarantees} \label{sec:detector_guarantees}
In this subsection, we describe our first main result which is concerned with the correctness of Algorithm~\ref{alg:detector}.
\begin{lemma} \label{lemma:detection_guarantee}
Let the linear dynamical system as defined in~\eqref{eq:vector_system_model}
be $2{k}$-sparse observable.
Consider a $k$-adversary satisfying Assumptions $1-5$ and a sensor subset $\mathbf{s} \subseteq\{1,2,\ldots, p\}$ with $|\mathbf{s}| \geq p - k$.
For any $\epsilon > 0$ and $\delta > 0$, there exists a large enough time window length $N$ such that when Algorithm~\ref{alg:detector} terminates with $\hat{d}_{\text{attack},\mathbf{s}}(t_1) = 0 $, the following probability bound holds:
\begin{align}
 \mathbb{P} \Big(  tr\left(\E_{t_1,N} \left(\e_{\mathbf{s}} \e^T_\mathbf{s}\right) -   \mathbf{P}^\star_{\mathbf{s}}   \right) \le  \epsilon \Big) \ge 1 - \delta,
 \label{eq:estimation_property}
\end{align}
where $\e_{\mathbf{s}}(t) = \x(t) - \hat{\x}_{\mathbf{s}}(t)$.
In other words, for large enough
$N$,
the bound
$\displaystyle{tr\left(\E_{t_1,N} \left(\e_{\mathbf{s}} \e^T_\mathbf{s}\right) -   \mathbf{P}^\star_{\mathbf{s}}   \right) \le  \epsilon}$
holds with high probability\footnote{
By stating that
the bound holds with high probability
for large enough $N$, we mean that
for any $\delta > 0$ and $\epsilon > 0$,
$\exists N_{\delta, \epsilon} \in \mathbb{N}$ such that
for $N > N_{\delta,\epsilon}$,
$\mathbb{P} \Big(  tr\left(\E_{t_1,N} \left(\e_{\mathbf{s}} \e^T_\mathbf{s}\right) -   \mathbf{P}^\star_{\mathbf{s}}   \right) \le  \epsilon \Big) \ge 1 - \delta$.}
(w.h.p.).
Moreover, if the attack is an $(\epsilon, \mathbf{s})$-effective attack,
the following also holds:
\begin{align}
 \mathbb{P} \left( \hat{d}_{\text{attack}, \mathbf{s}}(t_1)  = d_{\text{attack},\mathbf{s}}(t_1)  \right) \ge 1 - \delta,
\label{eq:detector_property} 
\end{align}
where $\hat{d}_{\text{attack}, \mathbf{s}}(t_1)$ is the output of Algorithm~\ref{alg:detector} while $d_{\text{attack},\mathbf{s}}(t_1)$ is the output of an oracle detector that knows 
the exact set of attacked sensors.
Hence,
Algorithm~\ref{alg:detector} can detect any
$(\epsilon,\mathbf{s})$-effective attack
w.h.p. for large enough $N$.
\label{lemma:detector_guarantee}
\end{lemma}
\begin{proof}[Proof of Lemma~\ref{lemma:detection_guarantee}]
We focus only on showing that~\eqref{eq:estimation_property}
holds whenever  Algorithm~\ref{alg:detector} terminates with $\hat{d}_{\text{attack},\mathbf{s}}(t_1) = 0$; the rest of the lemma follows easily from Definition~\ref{def:effective_attack}.
Since we assume that the set $\mathbf{s}$ has cardinality $\vert \mathbf{s} \vert \ge p - k$,
we can conclude that
there exists a subset $\mathbf{s}_g \subset \mathbf{s}$ with cardinality  $\vert \mathbf{s}_g \vert \ge p -2k$ sensors such that all its sensors are attack-free
(subscript $g$ in $\mathbf{s}_g$ stands for \textit{good} sensors in $\mathbf{s}$). 
Hence, by decomposing the set $\mathbf{s}$  into an attack-free set $\mathbf{s}_g$ and a potentially attacked set $\s \setminus \mathbf{s}_g$, we can conclude that, after a permutation similarity transformation for \eqref{eq:fast_block_residue_test_prediction},
the following holds for the attack-free subset $\mathbf{s}_g$:
\begin{align}
 \E_{N,t_1} \left(  \res_{\mathbf{s}_g}  \res_{\mathbf{s}_g}^T \right)  -\O_{\mathbf{s}_g}  \mathbf{P}^\star_{\mathbf{s}} \O_{\mathbf{s}_g}^T   - \M_{\mathbf{s}_g}   \preccurlyeq   \eta \; \mathbf{1}_{n(|\mathbf{s}|-k)} \mathbf{1}^T_{n(|\mathbf{s}|-k)} \nonumber .
\end{align}
Therefore,
{\allowdisplaybreaks[4]
\begin{align}
&tr \left(  \E_{N,t_1} \left(  \res_{\mathbf{s}_g}  \res_{\mathbf{s}_g}^T \right)
 -\O_{\mathbf{s}_g}  \mathbf{P}^\star_{\mathbf{s}} \O_{\mathbf{s}_g}^T  
 - \M_{\mathbf{s}_g}  \right) 
\nonumber \\
& \qquad \qquad \qquad \qquad \qquad \qquad   \leq    n(|\mathbf{s}|-k)  \eta 
 = \epsilon_1.
\label{eq:fast_block_residue_test_prediction_subset_trace}
\end{align}}
Similarly, after a suitable permutation $\Pi$,
we can decompose the block residue
$\res_{\mathbf{s}}(t)$ defined in equation~\eqref{eq:fast_residue_test} as follows:
{\allowdisplaybreaks[4]
\begin{align}
\Pi \left (\res_{\s}(t) \right) &= \matrix{\res_{\mathbf{s}_g}(t) \\ \res_{\mathbf{s}\setminus\mathbf{s}_g}(t) }
= \matrix{\overline{\mathbf{y}}_{\mathbf{s}_g}(t) - \O_{\mathbf{s}_g} \hat{\x}_{\mathbf{s}}(t) \\
\overline{\mathbf{y}}_{\mathbf{s}\setminus\mathbf{s}_g}(t) - \O_{\mathbf{s}\setminus\mathbf{s}_g} \hat{\x}_{\mathbf{s}}(t) } \nonumber \\
&= \matrix{\O_{\mathbf{s}_g} \x(t) + \J_{\mathbf{s}_g}  \bar{\mathbf{w}}(t) + \bar{\mathbf{v}}_{\mathbf{s}_g}(t) - \O_{\mathbf{s}_g}\hat{ \x}_{\s}(t)   \\ \overline{\mathbf{y}}_{\mathbf{s}\setminus\mathbf{s}_g}(t) - \O_{\mathbf{s}\setminus\mathbf{s}_g} \hat{\x}_{\mathbf{s}}(t)} \nonumber\\
&= \matrix{\O_{\mathbf{s}_g}\e_{\mathbf{s}}(t) +\mathbf{z}_{\mathbf{s}_g}(t)   \\ \overline{\mathbf{y}}_{\mathbf{s}\setminus\mathbf{s}_g}(t) - \O_{\mathbf{s}\setminus\mathbf{s}_g} \hat{\x}_{\mathbf{s}}(t) },
\label{eq:block_residue_expansion}
\end{align}}where $\mathbf{z}_{\mathbf{s}_g}(t) = \J_{\mathbf{s}_g}  \bar{\mathbf{w}}(t) + \bar{\mathbf{v}}_{\mathbf{s}_g}(t)$. Using \eqref{eq:block_residue_expansion}, 
we can rewrite $tr \left (\E_{N,t_1} \left( \res_{\mathbf{s}_g} \res_{\mathbf{s}_g}^T \right) \right)$
as:
{\allowdisplaybreaks[4]
\begin{align}
 & tr \left(  \E_{N,t_1} \left(  \res_{\mathbf{s}_g}  \res_{\mathbf{s}_g}^T \right) \right) 
\nonumber \\ & \qquad
 =
tr \left(  \O_{\mathbf{s}_g} \E_{N,t_1} \left( \mathbf{e}_{\mathbf{s}} \mathbf{e}_{\mathbf{s}}^T \right) \O_{\mathbf{s}_g}^T \right)
+  tr \left ( \E_{N,t_1} \left( \z_{\mathbf{s}_g} \z_{\mathbf{s}_g}^T \right) \right) 
\nonumber  \\
 & \qquad  \quad + 2 \E_{N,t_1} \left(  \mathbf{e}_{\mathbf{s}}^T  \O^T_{\mathbf{s}_g}  \z_{\mathbf{s}_g} \right).
\label{eq:residue_simplification_1}
\end{align}}
By combining 
\eqref{eq:fast_block_residue_test_prediction_subset_trace} and \eqref{eq:residue_simplification_1}:
\begin{align}
&  tr \left ( 
\O_{\mathbf{s}_g} \E_{N,t_1} \left( \mathbf{e}_{\mathbf{s}} \mathbf{e}_{\mathbf{s}}^T \right) \O_{\mathbf{s}_g}^T  
 - \O_{\mathbf{s}_g}  \mathbf{P}^\star_{\mathbf{s}} \O_{\mathbf{s}_g}^T \right)
 \nonumber \\
& \qquad \qquad \leq 
  tr \left ( \M_{\mathbf{s}_g}  \right)   
-  tr \left( \E_{N,t_1} \left( 
\z_{\mathbf{s}_g} \z_{\mathbf{s}_g}^T \right)  \right)
 +  \epsilon_1 
\nonumber \\ & \qquad \qquad \quad 
-  2\E_{N,t_1} \left(  \mathbf{e}_{\mathbf{s}}^T  \O^T_{\mathbf{s}_g}  \z_{\mathbf{s}_g} \right)  
\nonumber \\
& \qquad \qquad \stackrel{(a)}\leq 2 \epsilon_1  -  2\E_{N,t_1} \left(  \mathbf{e}_{\mathbf{s}}^T  \O^T_{\mathbf{s}_g}  \z_{\mathbf{s}_g} \right) \label{eq:LLN_z}\\
 & \qquad \qquad \stackrel{(b)}\leq 3 \epsilon_1,  \label{eq:3epsilon_1_bound}
\end{align}
where $(a)$ follows w.h.p. due to the law of large numbers (LLN)
for large enough $N$ (details in Appendix \ref{sec:appendix_LLN_proof}), and $(b)$ follows w.h.p. by showing that the cross term $2\E_{N,t_1} \left(  \mathbf{e}^T  \O^T_{\mathbf{s}_g}  \z_{\mathbf{s}_g} \right)$
has zero mean and vanishingly small variance for large enough $N$.
The cross term analysis is described in detail in Appendix \ref{sec:appendix_cross_term_analysis}.
Now recall that for any two matrices, $\mathbf{A}$ and $\mathbf{B}$ of
appropriate dimensions,
$tr \left ( \mathbf{AB} \right) = tr \left ( \mathbf{B} \mathbf{A} \right)$.
Using this fact along with \eqref{eq:3epsilon_1_bound},
the following holds:
\begin{align}
tr \left ( 
 \E_{N,t_1} \left(   \mathbf{e}_{\mathbf{s}} \mathbf{e}_{\mathbf{s}}^T -  \mathbf{P}^\star_{\mathbf{s}} \right)
\O^T_{\mathbf{s}_g}  \O_{\mathbf{s}_g}  
\right) 
&  \leq  3 \epsilon_1, 
\end{align} 
and hence, we get the following bound which completes the proof:
\begin{align}
 tr \left ( 
 \E_{N,t_1} \left( \mathbf{e}_{\mathbf{s}} \mathbf{e}_{\mathbf{s}}^T \right)  
 -  \mathbf{P}^\star_{\mathbf{s}} \right)
& \stackrel{(c)} \leq \frac{3 \epsilon_1}{\lambda_{min} \left (  \O^T_{\mathbf{s}_g}  \O_{\mathbf{s}_g} \right) }  \stackrel{(d)} \leq \frac{3 \epsilon_1}{\lambda_{\min,\mathbf{s} \setminus k}}
 \leq \epsilon \label{eq:detection_final_bound} 
\end{align} 
where $(c)$ follows from Lemma \ref{lemma:eigen_bound} in Appendix \ref{sec:trace_ineq} and $(d)$ follows from the definition of $\lambda_{\min,\mathbf{s} \setminus k}$.
Note that, it follows from $\vert \mathbf{s}_g \vert \ge p - 2k$ and $2k$-sparse observability, that both $\lambda_{min} \left (  \O^T_{\mathbf{s}_g}  \O_{\mathbf{s}_g} \right)$ and $\lambda_{\min,\mathbf{s} \setminus k}$ are bounded away from zero.
\end{proof}
\section{Effective Attack Detection and Secure State Estimation}
\label{sec:main_results}
Based on the performance guarantees for the \textsc{Attack-Detect} algorithm described
in Section~\ref{sec:detector},
in this section we describe our main results for Problems~\ref{prob:effective_attack} and~\ref{prob:optimal_sse}.
\subsection{Attack detection}
We start by showing a solution to
Problem~\ref{prob:effective_attack}
($\epsilon$-effective attack detection),
which follows directly from
Lemma~\ref{lemma:detector_guarantee}.
\begin{theorem}
Let the linear dynamical system defined in~\eqref{eq:vector_system_model}
be $k$-sparse observable system.
Consider a $k$-adversary satisfying Assumptions $1$-$5$,
and the detector
$\hat{d}_{\text{attack}}(t_1) = \textsc{Attack-Detect}(\mathbf{s}_{\text{all}}, t_1)$
where the set $\mathbf{s}_{\text{all}} = \{1,\ldots,p\}$.
Then, for large enough time window length $N$,
w.h.p. $\hat{d}_{\text{attack}}(t_1)$ is equal to the attack indicator
which solves Problem~\ref{prob:effective_attack}.
\label{th:detection}
\end{theorem}
\begin{proof}
The proof is similar to the proof of Lemma~\ref{lemma:detection_guarantee}.
In the proof of Lemma~\ref{lemma:detection_guarantee},
we basically required the set of good sensors $\s_g$ to
form an observable system.
Similarly, while checking for effective attacks on a sensor set of size $p$,
we require the set of good sensors (of size $ \geq p-k$) to form an observable system in order to repeat the
steps in the proof for Lemma~\ref{lemma:detection_guarantee};
this requirement is guaranteed by the $k$-sparse observability condition.
On a related note, in Section~\ref{sec:sparse_proof}, we give a coding theoretic interpretation for the $k$-sparse observability requirement for
attack detection.
\end{proof}
\subsection{Secure State Estimation}
Algorithm~\ref{alg:secure_state_estimation_vanilla}
describes our proposed solution for 
Problem~\ref{prob:optimal_sse}
(secure state estimation).
As described in
Algorithm~\ref{alg:secure_state_estimation_vanilla},
we exhaustively enumerate
$p \choose{p-k}$ sensor subsets of size $p - k$,
and then apply \textsc{Attack-Detect} on each sensor subset
until we find one subset $\mathbf{s}^*$ for which \textsc{Attack-Detect} returns $\hat{d}_{\text{attack},\mathbf{s}^*}(t_1) = 0 $
indicating that the subset is ($\epsilon$-effective) attack-free.
The following theorem states the performance
guarantees associated with
Algorithm~\ref{alg:secure_state_estimation_vanilla}.
\begin{theorem} \label{thm:ach}
Let the linear dynamical system defined in~\eqref{eq:vector_system_model}
be $2k$-sparse observable system.
Consider a $k$-adversary satisfying Assumptions 1-5.
Consider the state estimate $\hat{\x}_{\mathbf{s}^*}(t)$ computed by Algorithm~\ref{alg:secure_state_estimation_vanilla}.
\begin{algorithm}[t]
\caption{\textsc{ Exhaustive Search}}
\begin{algorithmic}[1]
\STATE Enumerate all sets $\mathbf{s} \in \mathbf{S} $ such that:\\ \mbox{$\qquad \qquad \mathbf{S} = \{ \mathbf{s} | \mathbf{s} \subset \{1,2, \ldots , p\}, \; |\mathbf{s}| = p-k \}$}.
\STATE  Exhaustively search for $\mathbf{s}^* \in \mathbf{S}$ for which $d_{\text{attack}, \mathbf{s}^* }(t_1) = 0$ and use $\hat{\mathbf{x}}_{\mathbf{s}^*}(t)$ for $t \in G$ as the state estimate.
\end{algorithmic}
\label{alg:secure_state_estimation_vanilla}
\end{algorithm}
Then, for any $\epsilon > 0$ and $\delta > 0$,
there exists a large enough $N$ such that:
\begin{align}
\mathbb{P} & \left ( 
  tr \left(\E_{N,t_1} \left( \mathbf{e}_{\mathbf{s}^*}\mathbf{e}_{\mathbf{s}^*}^T \right) \right)
 \le 
 tr \left(  \mathbf{P}^\star_{\mathbf{s}_{\text{worst},p - k}}  \right)  + \epsilon 
\right )
\geq 1- \delta,
 \label{eq:sse}
\end{align}
where $\mathbf{e}_{\mathbf{s}^*}(t) = \x(t) - \hat{\x}_{\mathbf{s}^*}(t)$
is the estimation error using $\hat{\x}_{\mathbf{s}^*}(t)$ as the state estimate.
In other words, w.h.p. Algorithm~\ref{alg:secure_state_estimation_vanilla}
achieves the bound
$\displaystyle{\limsup_{N\rightarrow \infty} 
\frac{1}{N} \sum_{t \in G } \mathbf{e}_{\s^*}^T(t)\mathbf{e}_{\s^*}(t) \le tr\left( \mathbf{P}^\star_{\mathbf{s}_{\text{worst},p - k}} \right)}$.
\end{theorem}
\begin{proof}
The result follows from Lemma~\ref{lemma:detector_guarantee} which ensures that, in the absence of the $(\epsilon,\mathbf{s})$-effective attack property, the calculated state estimate still guarantees the bound~\eqref{eq:estimation_property}.
This in turn implies that, in the worst case, $\displaystyle{\limsup_{N\rightarrow \infty} 
\frac{1}{N} \sum_{t \in G } \mathbf{e}_{{\s}^*}^T(t)\mathbf{e}_{\s^*}(t) = tr\left( \mathbf{P}^\star_{\mathbf{s}_{\text{worst},p - k}} \right)}$ is achievable.
However, since the $k$-adversary may not always attack the worst case set of sensors $\mathbf{s}_{\text{worst},p - k}$, we can replace the equality sign above with an inequality,
leading to $\displaystyle{\limsup_{N\rightarrow \infty} 
\frac{1}{N} \sum_{t \in G } \mathbf{e}_{{\s}^*}^T(t)\mathbf{e}_{\s^*}(t) \leq  tr\left( \mathbf{P}^\star_{\mathbf{s}_{\text{worst},p - k}} \right)}$.
\end{proof}
\section{Reducing Search Time
Using Satisfiability Modulo Theory Solving}
\label{sec:smt}

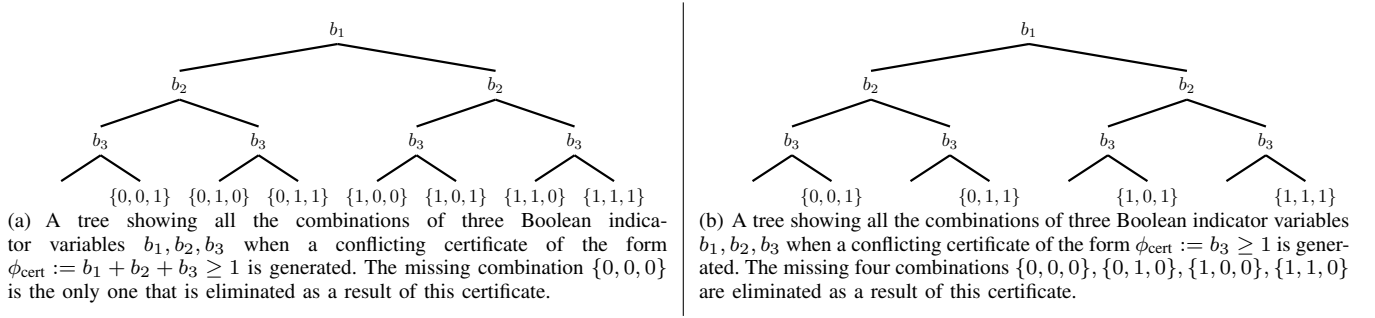
\begin{figure*}[t]
\centering
{
\begin{tabular}{c|c}
\subfigure[
A tree showing all the combinations of three Boolean indicator variables $b_1,b_2,b_3$ when a conflicting certificate of the form \mbox{$\phi_{\text{cert}} := b_1 + b_2 + b_3 \ge 1$} is generated. The missing combination $\{0,0,0\}$ is the only one that is eliminated as a result of this certificate.
]{\label{fig:cert1}
\resizebox{0.47\textwidth}{!}{
\begin{tikzpicture}
    \Tree 	[.$b_1$ 
    			[.$b_2$ 
				[.$b_3$ \phantom{$\{0,0,0\}$} $\{0,0,1\}$ ] 
				[.$b_3$ $\{0,1,0\}$ $\{0,1,1\}$ ] 
			] 
			[.$b_2$ 
				[.$b_3$ $\{1,0,0\}$ $\{1,0,1\}$ ]
				[.$b_3$ $\{1,1,0\}$ $\{1,1,1\}$ ] 
			]
		]
    ]
\end{tikzpicture} 
\vspace{5cm}
}
}
&
\subfigure[
A tree showing all the combinations of three Boolean indicator variables $b_1,b_2,b_3$ when a conflicting certificate of the form $\phi_{\text{cert}} := b_3 \ge 1$ is generated. The missing four combinations $\{0,0,0\}, \{0,1,0\}, \{1,0,0\}, \{1,1,0\}$ are eliminated as a result of this certificate.
]{\label{fig:cert2}
\resizebox{0.47\textwidth}{!}{
\begin{tikzpicture}
    \Tree 	[.$b_1$ 
    			[.$b_2$ 
				[.$b_3$ \phantom{$\{0,0,0\}$} $\{0,0,1\}$ ] 
				[.$b_3$ \phantom{$\{0,1,0\}$} $\{0,1,1\}$ ] 
			] 
			[.$b_2$ 
				[.$b_3$ \phantom{$\{1,0,0\}$} $\{1,0,1\}$ ]
				[.$b_3$ \phantom{$\{1,1,0\}$} $\{1,1,1\}$ ] 
			]
		]
    ]
\end{tikzpicture} 
}}
\end{tabular}
}
\caption{Pictorial example illustrating the effect of generating smaller conflicting certificates.}
\label{fig:certificate_reduction}
\end{figure*}

Algorithm~\ref{alg:secure_state_estimation_vanilla} exhaustively explores all
combinations of $p - k$ sensors until a set $\s^*$ satisfying $d_{\text{attack}, \s^*}(t_1) = 0$ is found.
In this section, we explore the idea of using sophisticated search techniques in order to harness the underlying combinatorial aspect of the secure state estimation problem.
In particular, we extend previous work by the authors and co-workers on using Satisfiability Modulo Theory (SMT)-like solvers~\cite{Yasser_SMT}, developed for the noiseless case, in order to improve the search time while preserving optimality of the solution.

The driving concept behind SMT solvers can be summarized as follows.
First, the search space of all sensor subsets with cardinality $p - k$,
is encoded using Boolean variables (the number of Boolean variables increases linearly with the number of sensors),
and a Boolean search engine (\emph{e.g.}, SAT solver) is used in order to traverse the search space.
Whenever the SAT solver suggests one possible solution in the search space, a higher level solver (typically referred to as the Theory-solver) is used to check the correctness of that particular solution. Finally, in order to prevent the SAT solver from enumerating all possible solutions in the search space, the Theory-solver generates counter examples (certificates), explaining why a particular solution is not valid. Each certificate is used by the SAT solver in order to prune the search space and hence enhance the performance of the overall algorithm. This methodology of ``counter-example guided search'' effectively breaks the secure state estimation problem into two simpler tasks over the Boolean and Reals domain.
Further details about this technique are described below.

\subsection{Overall Architecture}
We start by introducing a Boolean indicator variable $b = (b_1, \ldots, b_p) \in \B^p$ where the assignment $b_i = 1$ hypothesizes that the $i$th sensor is under attack while the
assignment
$b_i = 0$ hypothesizes that the $i$th sensor is attack-free. Using this indicator variable, $b$,
 we start by asking the \mbox{(pseudo-)Boolean} SAT solver to assign values to $b$ in order to satisfy the following formula:
\begin{align}
\phi(0) ::= \sum_{i = 1}^p b_i \le k, \label{eq:SMT_sum}
\end{align} 
which ensures that at most $k$ sensors are going to be hypothesized as being under attack
(the addition in \eqref{eq:SMT_sum} is over Reals).

In the next step, this  hypothesized assignment is then checked by the theory solver. This is done by running the \textsc{Attack-Detect} algorithm (Algorithm~\ref{alg:detector}) using only the set of hypothesized attack-free sensors $\mathbf{s}(b) = \{1, \ldots,p\} - \supp(b)$.
If the \textsc{Attack-Detect} algorithm returns $\hat{d}_{\text{attack}, \mathbf{s}(b)} = 0$
then our solver approves this hypothesis and the algorithm terminates. Otherwise, an UNSAT certificate (also known as a counter-example) is generated explaining why this assignment of  $b$ is not valid (\emph{i.e.}, a conflict). A trivial UNSAT certificate that can always be generated takes the following form (in iteration $j$):
\begin{align}
 \phi_{\text{cert}}(j) ::= \sum_{i \in \mathbf{s}(b)} b_i \ge 1 , 
\label{eq:trivial_cert}
\end{align}
which ensures that the current assignment of the variable $b$ is excluded. Once this UNSAT certificate is generated, the (pseudo-)Boolean SAT solver is then invoked again in the next iteration with the following constraints:
$$\phi(j+1) ::= \phi(j) \land \phi_{\text{cert}}(j),$$
until one assignment of the variable $b$ passes the attack detection test. This procedure is summarized in Algorithm~\ref{alg:smt}.

\subsection{Conflicting Certificates}
The generated UNSAT certificates heavily affect the overall execution time. Smaller UNSAT certificates prune the search space faster. For simplicity, consider the example shown in Figure~\ref{fig:certificate_reduction} where the vector $b$ has only three elements. On one hand, an UNSAT certificate that has the form $\phi_{\text{cert}} = b_1 + b_2 + b_3 \ge 1$ leads to pruning only one sample in the search space. On the other hand, a smaller UNSAT certificate that has the form $\phi_{\text{cert}} = b_1 \ge 1$ eliminates four samples in the search space which is indeed a higher reduction, and hence leads to better execution time.

To generate a compact (\emph{i.e.}, smaller) Boolean constraint that explains a conflict, we aim to find a small set of sensors that cannot all be attack-free.  To do so, we start by removing one sensor from the set $\mathbf{s}(b)$ and run the \textsc{Attack-Detect} algorithm on the reduced set $\mathbf{s}'(b)$ to obtain $\hat{d}_{\text{attack}, \mathbf{s}'(b)}$.
If $\hat{d}_{\text{attack}, \mathbf{s}'(b)}$ still equals one (which indicates that set $\s'(b)$ still contains a conflicting set of sensors), we generate the more compact certificate:
\begin{align}
 \phi_{\text{cert}}(j) ::= \sum_{i \in \mathbf{s}'(b)} b_i \ge 1.
\label{eq:conflict_cert}
\end{align}
We continue removing sensors one by one until we cannot find any more conflicting sensor sets.
Indeed, the order in which the sensors are removed is going to affect the overall execution time.
In Algorithm~\ref{alg:certificate1} we implement a heuristic (for choosing this order)
which is inspired by the strategy we adopted in the noiseless case \cite{Yasser_SMT}. 

Note that the reduced sets $\mathbf{s}'(b)$ are used only to generate the UNSAT certificates. Hence, it is direct to show that Algorithm~\ref{alg:smt} still preserves the optimality of the state estimate as stated by the following result.
\begin{theorem}
Let the linear dynamical system defined in~\eqref{eq:vector_system_model}
be $2k$-sparse observable system.
Consider a $k$-adversary satisfying Assumptions 1-5. Consider the state estimate $\hat{\x}_{\mathbf{s}^*}(t)$ computed by Algorithm~\ref{alg:smt}.
Then, for any $\epsilon > 0$ and $\delta > 0$,
there exists a large enough $N$ such that:
\begin{align}
\mathbb{P} & \left ( 
  tr \left(\E_{N,t_1} \left( \mathbf{e}_{\mathbf{s}^*}\mathbf{e}_{\mathbf{s}^*}^T \right) \right)
 \le 
 tr \left(  \mathbf{P}^\star_{\mathbf{s}_{\text{worst},p - k}}  \right)  + \epsilon 
\right )
\geq 1- \delta.
\end{align}
\end{theorem}

Note that although, for the sake of brevity, we did not analyze analytically the worst case execution time (in terms on number of iterations) of Algorithm~\ref{alg:smt}, we show numerical results in Section~\ref{sec:exp} that support the claim that the proposed SMT-like solver works much better in practice compared to the exhaustive search procedure (Algorithm~\ref{alg:secure_state_estimation_vanilla}).

\begin{algorithm}[t]
\caption{\textsc{SMT-based Search}}
\begin{algorithmic}[1]
\STATE status := \verb+UNSAT+;
\STATE $\phi_B :=  \sum_{i \in \{ 1,\hdots,p \}} b_i \le k $;
\WHILE{status == \texttt{UNSAT}}
	\STATE $b :=$ \textsc{SAT-Solve}$( \phi_B )$;
       \STATE $\s(b) := \{1,2,\ldots , p \} - \supp(b)$;
	\STATE  $ ( \hat{d}_{\text{attack}, \mathbf{s}(b)},  \{\hat{\x}_{\mathbf{s}(b)}(t)\}_{t \in G})$ \\  $:=\textsc{Attack-Detect}(\mathbf{s}(b), t_1)$;
	\IF {$\hat{d}_{\text{attack}, \mathbf{s}(b)}$ == 1}
		\STATE $\phi_{\text{cert}}$ \\ := \textsc{Generate-Certificate}$(\mathbf{s}(b), \{\hat{\x}_{\mathbf{s}(b)}(t) \}_{t \in G}  )$;
		\STATE $\phi_B := \phi_B \land \phi_{\text{cert}}$;
	\ENDIF
\ENDWHILE
\STATE $\mathbf{s}^* = \mathbf{s}(b)$;
\STATE \textbf{return} $ \{ \hat{\x}_{\mathbf{s}^*}(t) \}_{t \in G}$;
\end{algorithmic}
\label{alg:smt}
\end{algorithm}

\begin{algorithm}
\caption{\textsc{Generate-Certificate}$(\mathbf{s}, \{ \hat{\x}_{\mathbf{s}}(t) \}_{t \in G}  )$}
\begin{algorithmic}[1]
\STATE \textbf{Compute the residues for $i \in \s$} 
\STATE $\res_i(t) := \mathbf{y}_i(t) - \O_i \hat{\x}_{\mathbf{s}}(t)$, $\forall t \in G =\{t_1, \ldots , t_1 + N -1\}$ 
\STATE $\mu_i(t_1) := \left \vert tr \left (\E_{N,t_1} \left(  \res_i  \res_i^T \right)  -  \O_i  \mathbf{P}^\star_{\mathbf{s}} \O_i^T  -  \M_i \right)
- \eta n \right \vert $;
\STATE \textbf{Normalize the residues} 
\STATE $\mu_i(t_1) := \mu_i (t_1) / \lambda_{\max}\left ( \O_i^T \O_i \right)$, 
\STATE $\mathbf{\mu}(t_1) := \{ \mu_i (t_1) \}_{i \in \s}$ ;
\STATE \textbf{Sort the residues in ascending order} 
\STATE $\quad \mathbf{\mu}\_sorted(t_1) := \text{sortAscendingly}(\mathbf{\mu}(t_1))$;
\STATE \textbf{Choose sensor indices of $p-2k+1$ smallest residues} 
\STATE $\quad \mathbf{\mu}\_min\_r :=  \text{Index} \left ( \mu\_sorted [ 1 : p - 2k+1] \right )$
;
\STATE \textbf{Search linearly for the UNSAT certificate}
\STATE $ \text{status = \texttt{UNSAT}}; \; \text{counter} = 1; \; \phi_{\text{conf-cert}} = 1$;  \; $\s' = \s$
\WHILE{status == \texttt{UNSAT}}
	\STATE $\mathbf{s}' :=  \mathbf{s}' \setminus \mathbf{\mu}\_min\_r [\text{counter}]$;
	\STATE ($\hat{d}_{\text{attack}, \mathbf{s}'}, \{ \hat{\x}_{\mathbf{s}'}(t)\}_{t \in G}  $) := \textsc{Attack-Detect}$(\mathbf{s}', t_1)$;
	\IF{$\hat{d}_{\text{attack}, \mathbf{s}'}$ == 1}
		\STATE $\phi_{\text{conf-cert}}:= \phi_{\text{conf-cert}} \land \sum_{i \in \mathbf{s}'} b_i \ge 1$;
		\STATE counter := counter  + 1;	
	\ELSE
		\STATE status := \texttt{SAT};
	\ENDIF
\ENDWHILE
\STATE \textbf{return} $\phi_{\text{conf-cert}}$
\end{algorithmic}
\label{alg:certificate1}
\end{algorithm}

\section{Numerical Experiments}
\label{sec:exp}
In this section,
we report numerical results for
Algorithms~\ref{alg:secure_state_estimation_vanilla} and~\ref{alg:smt}
as described by the experiments below.
\subsection{Experiment 1: Residue test performance in Algorithm~\ref{alg:secure_state_estimation_vanilla}}
In this experiment,
we numerically check the performance of the residue test
involved in Algorithm~\ref{alg:secure_state_estimation_vanilla}
while checking for effective attacks across sensor subsets.
We generate a stable system randomly with $n=20$ (state dimension) and $p=5$ sensors. We select $k=2$ sensors at random,
and apply a random attack signal to the two sensors.
We apply Algorithm~\ref{alg:secure_state_estimation_vanilla} by running all the $\binom{5}{3} = 10$
Kalman filters (one for each distinct sensor subset of size $3$)
and do the residue test corresponding to each sensor subset.
Figure~\ref{fig:threshold}
shows the maximum entry in the
residue test matrix
$\mathbf{R}_{\s} = \E_{N,t_1} \left(  \res_{\mathbf{s}}  \res_{\mathbf{s}}^T \right)  - 
\left ( \O_{\mathbf{s}}  \mathbf{P}^\star_{\mathbf{s}} \O_{\mathbf{s}}^T   +
 \M_{\mathbf{s}} \right)$ for the $10$ different Kalman
filters.
It is apparent from Figure~\ref{fig:threshold}
that only one Kalman filter produces a state estimate that passes the residue test defined in Algorithm~\ref{alg:detector}.
This indeed corresponds to the set of attack-free sensors in the experiment.
\begin{figure}[!t]
\centering
{
\includegraphics[width=0.49 \textwidth]{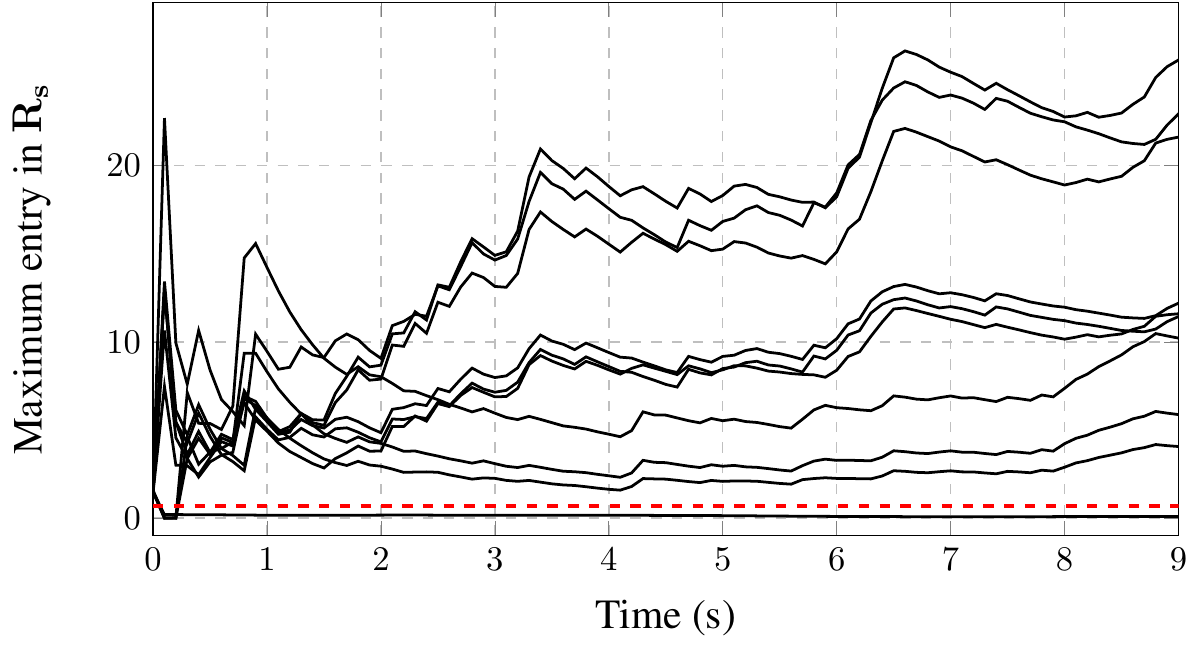}}
\caption{Figure showing the maximum entry in
the residue test matrix
$\mathbf{R}_{\s} = \E_{N,t_1} \left(  \res_{\mathbf{s}}  \res_{\mathbf{s}}^T \right)  - 
\left ( \O_{\mathbf{s}}  \mathbf{P}^\star_{\mathbf{s}} \O_{\mathbf{s}}^T   +
 \M_{\mathbf{s}} \right)
$
for the $10$
Kalman filters in Experiment $1$ versus the threshold $\eta=0.7$ (indicated by the dashed red line).
As shown in the figure,
there is only one subset of sensors which satisfies the threshold $\eta$,
and this corresponds to the attack-free set of sensors.}
\label{fig:threshold}
\end{figure}

\subsection{Experiment 2: Performance of SMT-based Search}
In this experiment, we compare the sensor subset search time for the SMT-based approach (Algorithm~\ref{alg:smt}) with that for the
exhaustive search approach (Algorithm~\ref{alg:secure_state_estimation_vanilla}).
For this experiment, we fix $n  = 50$ (state dimension) and
vary the number of sensors from $p = 3$ to $p = 15$.
For each system, we pick one third of the sensors to be under attack, \emph{i.e.},
$k = \lfloor p/3 \rfloor$.
The attack signal is chosen as a linear function of the measurement noise.
For each system, we run the bank of $\binom{p}{p - k}$ Kalman filters to generate the state estimates corresponding to all
sensor subsets of size $p-k$.
We then use both exhaustive search as well as the SMT-based search to find the
sensor subset
that satisfies the residue test
in Algorithm~\ref{alg:detector}.
Figure~\ref{fig:smt_time} shows the average time needed to perform the search across 50 runs of the same experiment.
Figure~\ref{fig:smt_time} suggests that the SMT-based search has an exponential improvement over exhaustive search
as the number of sensors increases.
In particular, for $p=15$, the SMT-based search out-performs
exhaustive search by an order of magnitude.

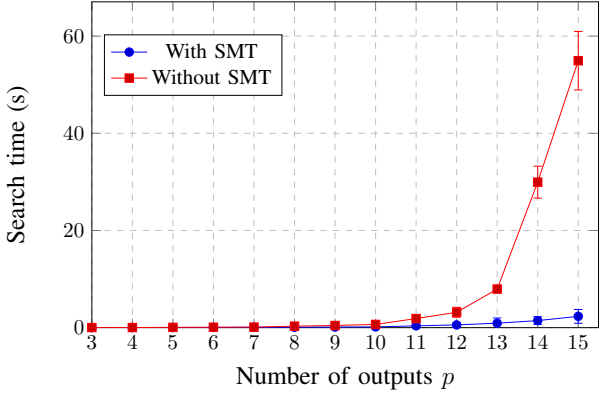
\begin{figure}[!t]
\centering
{
\begin{tabular}{c c} 
	\resizebox{0.45 \textwidth}{!}{
	\begin{tikzpicture}
		\begin{axis}[
		width = 10cm,
		height = 7cm,
	    	xlabel=\large{Number of outputs $p$},
	    	ylabel=\large{Search time (s)},
	    	xtick={3,4,5,6,7,8,9,10,11,12,13,14,15},
		xmin = 3,
		xmax = 15.5,
		ymin = 0,
	    	legend style={at={(0.2,0.9)},anchor=north,legend columns=1},
	    	ymajorgrids=true,
		xmajorgrids=true,
	    	grid style=dashed,
		]
		\addplot+[error bars/.cd,y dir=both,y explicit] 
			coordinates {	(3, 0.0219) +- (0, 0.0012)
						(4, 0.0274) +- (0, 0.0035)
						(5, 0.0434) +- (0, 0.00112)
						(6, 0.0493) +- (0, 0.0093)
						(7, 0.0748) +- (0, 0.017)
						(8, 0.0930) +- (0, 0.0574)
						(9,0.1278) +- (0, 0.088)
						(10, 0.1608) +- (0,0.0862) 
						(11, 0.3763) +- (0, 0.108)
						(12, 0.5658) +- (0, 0.427)
						(13, 0.9209) +- (0, 1.073)
						(14, 1.4554) +- (0, 0.8079)
						(15, 2.32837016) +- (0,1.42) 
			};
		\addplot+[error bars/.cd,y dir=both,y explicit] 
			coordinates {	
						(3, 0.0206) +- (0, 0.0015)
						(4, 0.0252) +- (0, 0.0055)
						(5, 0.0621) +- (0,0.0047)
						(6, 0.0893) +- (0, 0.0105)
						(7, 0.1148) +- (0, 0.0086)
						(8, 0.3065) +- (0, 0.0374)
						(9,0.4512) +- (0, 0.088)
						(10, 0.6726) +- (0,0.1543) 
						(11, 1.8763) +- (0, 0.0325)
						(12, 3.1658) +- (0, 1.0127)
						(13, 7.9419) +- (0, 0.4927)
						(14, 29.9454) +- (0, 3.2809)
						(15, 54.9367) +- (0,6.0236) 
			};
		\legend{With SMT, Without SMT}
		\end{axis}
	\end{tikzpicture} }
\end{tabular}
}	
\caption{Comparison of sensor subset search times for exhaustive search and SMT based search.}
\label{fig:smt_time}
\end{figure}

\section{Sparse observability: Coding theoretic view} \label{sec:sparse_proof}
In this section, we revisit the sparse observability condition
against a $k$-adversary and give a coding theoretic interpretation for the same.
We demonstrate how techniques developed for
error correction in classical coding theory \cite{blahut}
can be used for understanding the resilience
of dynamical systems against malicious
attacks\footnote{With a similar motivation,
in \cite{shaunak_CDC14} the authors use coding techniques to
enable secure state estimation in the presence of a corrupt observer with
unattacked sensor outputs.}.
We first describe our coding theoretic interpretation
for sensor attacks in a linear system,
and then discuss how it can be generalized for non-linear systems.

\par Consider the linear dynamical system in \eqref{eq:vector_system_model} without the process and sensor noise (\emph{i.e.}, $\mathbf{x}\left(t+1\right) = \mathbf{Ax}(t), \; \mathbf{y}(t) = \mathbf{C} \mathbf{x}(t)+ \mathbf{a}(t)$).
If the system's initial state is $\mathbf{x}(0) \in \mathbb{R}^n$ and the system is $\theta$-sparse observable, then clearly in the absence of sensor attacks, by observing the outputs from any $p-\theta$
sensors for $n$ time instants ($t =0,1,\ldots , n-1$) we can exactly recover $\mathbf{x}(0)$ and hence, \textit{exactly} estimate the state of the plant.
A coding theoretic view of this can be given as follows.
Consider the outputs from sensor $d \in \{1,2,\ldots, p\}$ for $n$ time instants as a symbol $\pazocal{Y}_d \in \mathbb{R}^n$.
Thus, in the (symbol) observation vector $\pazocal{Y}= \begin{bmatrix}\pazocal{Y}_1 & \pazocal{Y}_2  \ldots \pazocal{Y}_p \end{bmatrix}$,
due to $\theta$-sparse observability,
any $p-\theta$ symbols are sufficient (in the absence of attacks) to recover the initial state $\mathbf{x}(0)$.
Now, let us consider the case of a $k$-adversary which can arbitrarily corrupt any $k$ sensors.
In the coding theoretic view, this corresponds to arbitrarily corrupting any $k$ (out of $p$) symbols in the observation vector.
Intuitively,
based on the relationship between error correcting codes and the Hamming distance between codewords in classical coding theory \cite{blahut}, one can expect the recovery of the initial state despite such corruptions to depend on the (symbol) Hamming distance
between the observation vectors corresponding to two distinct initial states (say $\mathbf{x}^{(1)}(0)$ and $\mathbf{x}^{(2)}(0)$ with $\mathbf{x}^{(1)}(0) \neq \mathbf{x}^{(2)}(0)$).
In this context,
the following lemma relates $\theta$-sparse observability
to the minimum Hamming distance between observation vectors in the absence of attacks.
\begin{lemma} \label{lemma:distance}
For a $\theta$-sparse observable system, the minimum (symbol) Hamming distance between observation vectors corresponding to distinct initial states is $\theta+1$. 
\end{lemma}

\begin{proof}
Consider a system with $p$ sensors, and observation vectors $\pazocal{Y}^{(1)}$ and $\pazocal{Y}^{(2)}$ corresponding to distinct initial states $\mathbf{x}^{(1)}(0)$ and $\mathbf{x}^{(2)}(0)$. 
Due to $\theta$-sparse observability,
at most $p-\theta-1$ symbols in $\pazocal{Y}^{(1)}$ and $\pazocal{Y}^{(2)}$ can be identical;
if any $p-\theta$ of the symbols are identical, this would imply $\mathbf{x}^{(1)}(0) = \mathbf{x}^{(2)}(0)$.
Hence, the (symbol) Hamming distance between the observation vectors
$\pazocal{Y}^{(1)}$ and $\pazocal{Y}^{(2)}$
(corresponding to $\mathbf{x}^{(1)}(0)$ and $\mathbf{x}^{(2)}(0)$)
is at least $p-(p-\theta-1) = \theta +1$ symbols.
Also, there exists a pair of initial states $\left( \mathbf{x}^{(1)}(0), \mathbf{x}^{(2)}(0)\right)$,
such that the corresponding observation vectors $\pazocal{Y}^{(1)}$ and $\pazocal{Y}^{(2)}$ are identical in exactly $p-\theta-1$ symbols\footnote{If there is no such pair of initial states,
the initial state can be recovered by observing any $p-\theta-1$ sensors.
By definition, in a $\theta$-sparse observable system, $\theta$ is the largest positive integer, such that the initial state can be recovered by observing any $p-\theta$ sensors.}
and differ in the rest $\theta+1$ symbols.
Hence, the minimum (symbol) Hamming distance between the observation vectors is  $\theta+1$.
\end{proof}

\par
For a $\theta$-sparse observable system,
since the minimum Hamming distance between the observation vectors corresponding to
distinct initial states is $\theta+1$, we can:
\begin{itemize}
\item[(1)] correct up to $k < \frac{\theta+1}{2}$ sensor corruptions,
\item[(2)] detect up to $k \leq \theta $ sensor corruptions.
\end{itemize}
Note that
(1) above is equivalent to
$2k \leq \theta$ (sparse observability condition
for secure state estimation \cite{YasserETPGarXiv}).
It should be noted that a $k$-adversary can attack \textit{any} set of $k$ (out of $p$) sensors,
and the condition $k < \frac{\theta+1}{2}$ is both necessary and sufficient for exact state estimation despite such attacks.
When $k \geq \frac{\theta+1}{2}$,
it is straightforward to show a scenario where the observation vector (after attacks) can be explained by multiple initial states,
and hence exact state estimation is not possible.
The following example illustrates such an attack scenario.

\begin{example} \label{example:coding_interpretation}
Consider a $\theta$-sparse observable system with $\theta=3$, number of sensors $p=5$, and a $k$-adversary with $k=2$.
Clearly, the condition $k < \frac{\theta+1}{2}$ is not satisfied in this example.
Let $\mathbf{x}^{(1)}(0)$ and $\mathbf{x}^{(2)}(0)$ be distinct initial states,
such that the corresponding observation vectors $\pazocal{Y}^{(1)}$ and $\pazocal{Y}^{(2)}$ have (minimum) Hamming distance $\theta+1 = 4$ symbols.
Figure~\ref{fig:coding_example} depicts the observation vectors $\pazocal{Y}^{(1)}$ and $\pazocal{Y}^{(2)}$, and for the sake of this example,
we assume that the observation vectors have the same first symbol (i.e., $\pazocal{Y}^{(1)}_1 = \pazocal{Y}_1^{(2)} = \pazocal{Y}_1$) and differ in the rest $4$ symbols
(hence, a Hamming distance of $4$).
\begin{figure}[!ht]
\begin{center}
\includegraphics[scale=0.5]{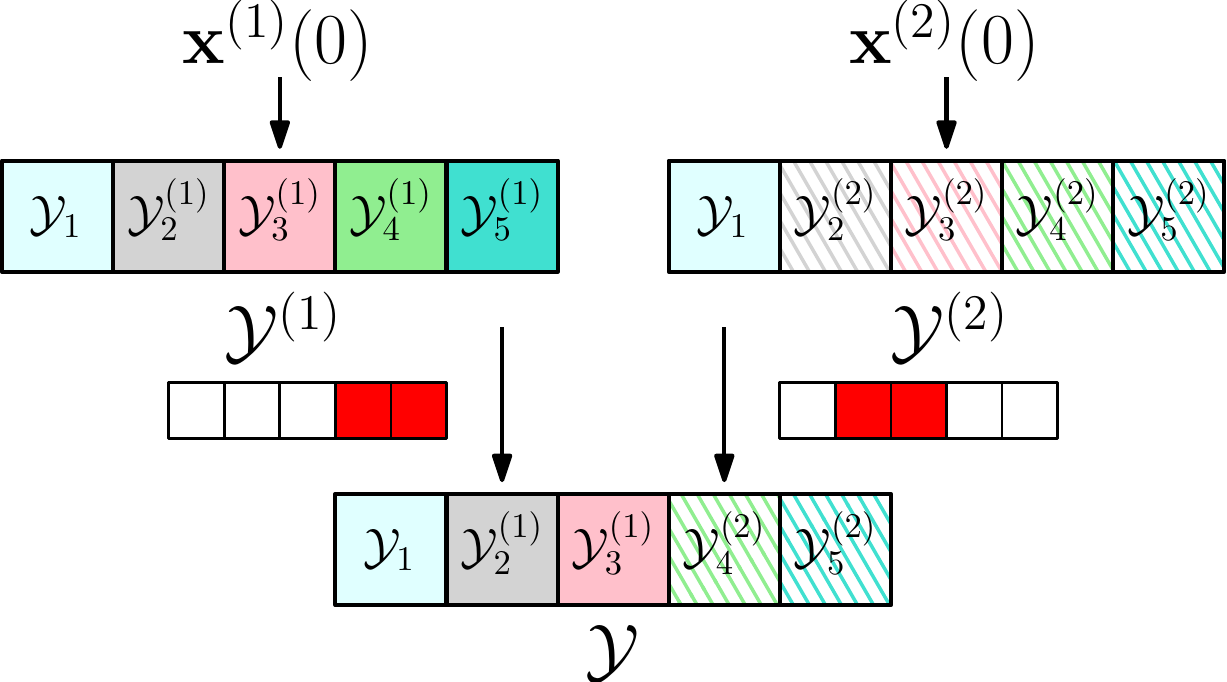}
\caption{Example with $\theta = 3$, $p=5$ and $k=2$.
For distinct initial states $\mathbf{x}^{(1)}(0)$ and $\mathbf{x}^{(2)}(0)$, the corresponding observation vectors are $\pazocal{Y}^{(1)}$ and $\pazocal{Y}^{(2)}$.
Given (attacked) observation vector $\pazocal{Y} = \left [ \pazocal{Y}_1 \; \;  \pazocal{Y}_2^{(1)} \; \; \pazocal{Y}_3^{(1)} \; \;  \pazocal{Y}_4^{(2)} \; \; \pazocal{Y}_5^{(2)} \right]$, there are two possibilities for the initial state: (a) $\mathbf{x}^{(1)}(0)$ with attacks on sensors
$4$ and $5$, or (b) $\mathbf{x}^{(2)}(0)$ with attacks on sensors
$2$ and $3$.}
\label{fig:coding_example}
\end{center}
\end{figure}
Now, as shown in Figure~\ref{fig:coding_example}, suppose the observation vector after attacks was
$\pazocal{Y} = \left [ \pazocal{Y}_1 \; \;  \pazocal{Y}_2^{(1)} \; \; \pazocal{Y}_3^{(1)} \; \;  \pazocal{Y}_4^{(2)} \; \; \pazocal{Y}_5^{(2)} \right]$.
Clearly, there are two possible explanations for this (attacked) observation vector: (a) the initial state was $\mathbf{x}^{(1)}(0)$ and sensors $4$ and $5$ were attacked,
or (b) the initial state was $\mathbf{x}^{(2)}(0)$ and sensors $2$ and $3$ were attacked. Since there are two possibilities,
we cannot estimate the initial state exactly given the attacked observation vector.
This example can be easily generalized to show the necessity of the condition $k < \frac{\theta+1}{2}$.
\end{example}

\par For (noiseless) non-linear systems, by analogously defining $\theta$-sparse observability, the same coding theoretic interpretation holds.
This leads to the necessary and sufficient conditions for attack detection and secure state estimation in any noiseless dynamical system with sensor attacks.

\bibliographystyle{IEEEtran}
\bibliography{bibliography2} 

\appendix
\subsection{Proof details for Theorem~\ref{th:detection}}

\subsubsection{Proof of \eqref{eq:LLN_z} using LLN} \label{sec:appendix_LLN_proof}
{\allowdisplaybreaks[4]
\begin{align}
& tr \left( \mathbf{M}_{\mathbf{s}_g} \right) -   tr \left ( 
\E_{N,t_1}    \left (  \mathbf{z}_{\mathbf{s}_g} \mathbf{z}^T_{\mathbf{s}_g}\right) 
\right)
 \nonumber \\
&= \frac{1}{n}  \sum_{l=0}^{n-1}    
tr \left( \mathbf{M}_{\mathbf{s}_g}  \right) -  
\frac{1}{N}  \sum_{t \in G}   tr \left (   \mathbf{z}_{\mathbf{s}_g}(t)  
 \mathbf{z}^T_{\mathbf{s}_g}(t)    \right) 
  \nonumber \\
& \stackrel{(a)}=   \sum_{l=0}^{n-1}   \frac{1}{n} \left (
tr \left( \mathbf{M}_{\mathbf{s}_g} \right) -  
\frac{1}{N_B}  \sum_{t \in G_l}   
tr \left (   \mathbf{z}_{\mathbf{s}_g}(t)
    \mathbf{z}^T_{\mathbf{s}_g}(t)    \right) 
 \right)  \nonumber \\
& \leq   \sum_{l=0}^{n -1 }  \frac{1}{n }  \left | 
tr \left( \mathbf{M}_{\mathbf{s}_g} \right) -  
\frac{1}{N_B}  \sum_{t \in G_l}   tr \left (    \mathbf{z}_{\mathbf{s}_g} (t)  
\mathbf{z}^T_{\mathbf{s}_g} (t)    \right)  
 \right|   \nonumber \\
& \stackrel{(b)}\leq  \epsilon_1,   \nonumber
\end{align}}
where (a) follows from partitioning time window $G$ (of size $N$)
into $n$ groups $G_0, G_1, \ldots G_{n-1}$ (each of size $N_B$)
such that $G_l = \{ t |  \left( \left( t-t_1 \right) \; mod \; n \right) = l \}$,
and (b) follows w.h.p. from LLN (for different time indices in $G_l$, $ tr \left (  \mathbf{z}_{\mathbf{s}_g}(t)   \mathbf{z}^T_{\mathbf{s}_g}(t) \right)$
corresponds to i.i.d. realizations of the same random variable).
\subsubsection{Cross term analysis and proof of \eqref{eq:3epsilon_1_bound}} \label{sec:appendix_cross_term_analysis}
The cross term $ 2 \E_{N,t_1} \left(  \mathbf{z}^T_{\mathbf{s}_g}  \O_{\mathbf{s}_g}  \mathbf{e}_{\s} \right) $ can be written down as a sum of $n$ terms as shown below:
{\allowdisplaybreaks[4]\begin{align}
2 \E_{N,t_1} \left(  \mathbf{z}^T_{\mathbf{s}_g} 
 \O_{\mathbf{s}_g}  \mathbf{e}_{\s} \right) 
&\stackrel{(a)} = \frac{2}{n} \sum_{l=0}^{n -1 } 
\left( \frac{1}{N_B} \sum_{t \in G_l}  \mathbf{z}^T_{\mathbf{s}_g}
(t)   \O_{\mathbf{s}_g}  \mathbf{e}_{\s}(t)   \right)  \nonumber\\
&= \frac{2}{n} \sum_{l=0}^{n - 1}  \zeta_l, 
\end{align}}
where (a) follows from partitioning time window $G$ (of size $N$)
into $n$ groups $G_0, G_1, \ldots G_{n-1}$ (each of size $N_B$)
such that $G_l = \{ t |  \left( \left( t-t_1 \right) \; mod \; n \right) = l \}$.
Now, we will show that each $ \zeta_l$ has zero mean and vanishingly small variance for large enough $N$.
The mean analysis can be done as shown below:
{\allowdisplaybreaks[4]\begin{align}
\mathbb{E} \left (  \zeta_l \right) 
 & = \mathbb{E} \left (   \frac{1}{N_B} \sum_{t \in G_l} \mathbf{z}^T_{\mathbf{s}_g}(t)   \O_{\mathbf{s}_g}  \mathbf{e}_{\s}(t)  \right)  \nonumber \\
 & \stackrel{(a)}=  \frac{1}{N_B} \sum_{t \in G_l}  
\mathbb{E} \left ( \mathbf{z}^T_{\mathbf{s}_g}(t)   \right) 
\mathbb{E}   \left (  \O_{\mathbf{s}_g}  \mathbf{e}_{\s}(t)  \right)  = 0 \label{eq:zero_mean_vector_prediction_fast},
\end{align}}
where (a) follows from the independence of $\mathbf{e}_{\s}(t)$ from $\mathbf{z}^T_{\mathbf{s}_g}(t) $ (due to assumptions \ref{ass:A1} and \ref{ass:A2}).
This implies that the cross term
$ 2 \E_{N,t_1} \left(  \mathbf{z}^T_{\mathbf{s}_g}  \O_{\mathbf{s}_g}  \mathbf{e}_{\s} \right) $
has zero mean.
As a consequence of \eqref{eq:zero_mean_vector_prediction_fast} and \eqref{eq:LLN_z},
{\allowdisplaybreaks[4]\begin{align}
 2 \epsilon_1 
& \geq 
 \mathbb{E} \left(  
 \E_{N,t_1}  \left ( tr  \left(  
 \O_{\mathbf{s}_g}   \left( \mathbf{e}_{\s}  \mathbf{e}_{\s}^T  - \mathbf{P}^\star_{\mathbf{s}}  \right) 
\O^T_{\mathbf{s}_g}  
\right )
\right)  
\right) 
 \nonumber \\
& =
\mathbb{E} \left(  
 \E_{N,t_1}  \left ( tr  \left(   \left( \mathbf{e}_{\s}  \mathbf{e}_{\s}^T  - \mathbf{P}^\star_{\mathbf{s}} \right) 
\O^T_{\mathbf{s}_g}  \O_{\mathbf{s}_g} \right) \right ) \right) 
\nonumber\\
& \stackrel{(a)} \geq 
\lambda_{min} \left(  \O^T_{\mathbf{s}_g} \O_{\mathbf{s}_g}  \right)
\E  \left(  \E_{N,t_1} \left(     tr   \left( \mathbf{e}_{\s}  \mathbf{e}_{\s}^T  - \mathbf{P}^\star_{\mathbf{s}} \right)  \right ) \right),
\label{eq:expectation_bound_general_case_prediction_fast_0}
\end{align}}
where (a) follows from Lemma~\ref{lemma:eigen_bound} (discussed in Appendix~\ref{sec:trace_ineq}).
Using \eqref{eq:expectation_bound_general_case_prediction_fast_0},
\begin{align}
\E \left ( 
\E_{N, t_1} 
 \left(     tr   \left(  \mathbf{e}_{\s}  \mathbf{e}_{\s}^T \right)  \right)
\right)
& \leq  \frac{2 \epsilon_1}{\lambda_{min} \left(  \O^T_{\mathbf{s}_g} \O_{\mathbf{s}_g}  \right)}  
+ tr  \left( \mathbf{P}^\star_{\mathbf{s}} \right) \label{eq:expectation_bound_general_case_prediction_fast}.
\end{align}
We will use the intermediate result
\eqref{eq:expectation_bound_general_case_prediction_fast} in the variance analysis
of $ \displaystyle{ \zeta_l =  \frac{1}{N_B} \sum_{t \in G_l} \mathbf{z}^T_{\mathbf{s}_g}(t)   \O_{\mathbf{s}_g}  \mathbf{e}_{\s}(t)   = 
 \frac{1}{N_B} \sum_{t \in G_l} \mathbf{e}_{\s}^T(t)  \O^T_{\mathbf{s}_g}   
\mathbf{z}_{\mathbf{s}_g}(t)}$ described below.
\par For any $\epsilon_2 > 0$, there exists a large enough $N_B$ such that:
{\allowdisplaybreaks[4]
\begin{align}
& Var \left( \frac{1}{N_B} \sum_{t \in G_l} \mathbf{e}_{\s}^T(t) \O^T_{\mathbf{s}_g}  \mathbf{z}_{\mathbf{s}_g}(t) \right) \nonumber \\
& = \mathbb{E} \left ( \left( \frac{1}{N_B} \sum_{t \in G_l}
 \mathbf{e}_{\s}^T(t) \O^T_{\mathbf{s}_g}  \mathbf{z}_{\mathbf{s}_g}(t)
 \right)^2 \right)  \nonumber \\
& \quad  
- \left (\mathbb{E}\left( \frac{1}{N_B} \sum_{t \in G_l} \mathbf{e}_{\s}^T(t) \O^T_{\mathbf{s}_g} \mathbf{z}_{\mathbf{s}_g}(t) \right) \right)^2 \nonumber \\
& \stackrel{(a)}= \mathbb{E} \left ( 
\left( \frac{1}{N_B} \sum_{t \in G_l} \mathbf{e}_{\s}^T(t) \O^T_{\mathbf{s}_g} 
\mathbf{z}_{\mathbf{s}_g}(t)
\right)^2 \right)  \nonumber \\
& = \mathbb{E}\left( \frac{1}{N^2_B} \sum_{t \in G_l} \mathbf{e}_{\s}^T(t) \O^T_{\mathbf{s}_g} \mathbf{z}_{\mathbf{s}_g}(t)
 \mathbf{e}_{\s}^T(t) \O^T_{\mathbf{s}_g} \mathbf{z}_{\mathbf{s}_g}(t) \right) 
 \nonumber \\ & \quad
 + \mathbb{E}\left( \frac{2}{N^2_B} \sum_{t,t' \in G_l, \; t < t'} \mathbf{e}_{\s}^T(t) \O^T_{\mathbf{s}_g} \mathbf{z}_{\mathbf{s}_g}(t)
 \mathbf{e}_{\s}^T(t') \O^T_{\mathbf{s}_g} \mathbf{z}_{\mathbf{s}_g}(t')
   \right) \nonumber\\
& \stackrel{(b)} = \mathbb{E}\left( \frac{1}{N^2_B} \sum_{t \in G_l} \mathbf{e}_{\s}^T(t) \O^T_{\mathbf{s}_g} \mathbf{z}_{\mathbf{s}_g}(t)
 \mathbf{e}_{\s}^T(t) 
\O^T_{\mathbf{s}_g} \mathbf{z}_{\mathbf{s}_g}(t) \right) 
\nonumber \\ & \;\;
+
 \frac{2}{N^2_B} \sum_{t,t' \in G_l, \; t < t'} \mathbb{E}\left( \mathbf{e}_{\s}^T(t) \O^T_{\mathbf{s}_g} \mathbf{z}_{\mathbf{s}_g}(t)
 \mathbf{e}_{\s}^T(t') \O^T_{\mathbf{s}_g} \right)
\mathbb{E}\left( \mathbf{z}_{\mathbf{s}_g}(t')  \right) \nonumber\\
& \stackrel{(c)}= \mathbb{E}\left( \frac{1}{N^2_B} \sum_{t \in G_l} \mathbf{e}_{\s}^T(t) \O^T_{\mathbf{s}_g} \mathbf{z}_{\mathbf{s}_g}(t) \mathbf{e}_{\s}^T(t) \O^T_{\mathbf{s}_g} \mathbf{z}_{\mathbf{s}_g}(t)  \right) \nonumber\\
& = \mathbb{E}\left( \frac{1}{N^2_B} \sum_{t \in G_l} \mathbf{e}_{\s}^T(t) 
\O^T_{\mathbf{s}_g} \mathbf{z}_{\mathbf{s}_g}(t)
 \left( \O^T_{\mathbf{s}_g} \mathbf{z}_{\mathbf{s}_g}(t) 
\right)^T \mathbf{e}_{\s}(t) \right) \nonumber\\
& \stackrel{(d)}= \mathbb{E} \left( \frac{1}{N^2_B} \sum_{t \in G_l} tr \left( \mathbf{e}_{\s}^T(t) \O^T_{\mathbf{s}_g} 
\mathbf{z}_{\mathbf{s}_g}(t) 
\left( \O^T_{\mathbf{s}_g} 
\mathbf{z}_{\mathbf{s}_g}(t) \right)^T \mathbf{e}_{\s}(t) \right) \right) \nonumber\\
& = \mathbb{E} \left( \frac{1}{N^2_B} \sum_{t \in G_l} tr \left( \O^T_{\mathbf{s}_g} \mathbf{z}_{\mathbf{s}_g}(t) 
 \left( \O^T_{\mathbf{s}_g} \mathbf{z}_{\mathbf{s}_g}(t) 
 \right)^T \mathbf{e}_{\s}(t) \mathbf{e}_{\s}^T(t) \right) \right) \nonumber\\
& \stackrel{(e)}= \frac{ 1}{N^2_B}\sum_{t \in G_l} tr \left( \mathbb{E} \left( \O^T_{\mathbf{s}_g}
\mathbf{z}_{\mathbf{s}_g}(t)  \mathbf{z}^T_{\mathbf{s}_g}(t) \O_{\mathbf{s}_g}  \right)
\mathbb{E}
\left(\mathbf{e}_{\s}(t) \mathbf{e}_{\s}^T(t) \right) \right) \nonumber\\
& \stackrel{(f)}\leq \frac{1}{N^2_B} \sum_{t \in G_l} \lambda_{max} \left( \O^T_{\mathbf{s}_g} \mathbf{M}_{\mathbf{s}_g} \O_{\mathbf{s}_g} \right)
 tr \left( \mathbb{E}
\left(\mathbf{e}_{\s}(t) \mathbf{e}_{\s}^T(t) \right) \right) \nonumber\\
& = \frac{ \lambda_{max} \left( \O^T_{\mathbf{s}_g} \mathbf{M}_{\mathbf{s}_g} \O_{\mathbf{s}_g} \right) }{N_B} \mathbb{E}
\left( \frac{1}{N_B} \sum_{t \in G_l} tr \left( \mathbf{e}_{\s}(t) \mathbf{e}_{\s}^T(t) \right) \right) \nonumber\\
& = \frac{ \lambda_{max} \left( \O^T_{\mathbf{s}_g} \mathbf{M}_{\mathbf{s}_g} \O_{\mathbf{s}_g} \right) }{N_B} \mathbb{E}
\left( \frac{1}{N_B} \sum_{t \in G_l} \mathbf{e}_{\s}^T(t) \mathbf{e}_{\s}(t) \right) \nonumber\\
& = \frac{ n  \lambda_{max} \left( \O^T_{\mathbf{s}_g} \mathbf{M}_{\mathbf{s}_g} \O_{\mathbf{s}_g} \right) }{N_B} \mathbb{E}
\left( \frac{1}{N} \sum_{t \in G_l} \mathbf{e_{\s}}^T(t) \mathbf{e}_{\s}(t) \right) \nonumber\\
& \leq \frac{ n \lambda_{max} \left( \O^T_{\mathbf{s}_g} \mathbf{M}_{\mathbf{s}_g} \O_{\mathbf{s}_g} \right) 
\E \left(  \frac{1}{N} \sum_{t \in G} \mathbf{e}_{\s}^T(t) \mathbf{e}_{\s}(t) \right) }
{N_B} \nonumber\\
& \stackrel{(g)} \leq \epsilon_2,
\end{align}}
where (a) follows from \eqref{eq:zero_mean_vector_prediction_fast},
(b) follows from the independence of $ \mathbf{z}_{\mathbf{s}_g}(t')$
from $ \mathbf{e}_{\s}^T(t) \O^T_{\mathbf{s}_g} \mathbf{z}_{\mathbf{s}_g}(t)
\mathbf{e}_{\s}^T(t') \O^T_{\mathbf{s}_g}$ for $t'>t$,
(c) follows from $\mathbb{E}\left( \mathbf{z}_{\mathbf{s}_g}(t') \right) =\mathbf{0}$,
(d) follows from $\mathbf{e}_{\s}^T(t) \O^T_{\mathbf{s}_g} \mathbf{z}_{\mathbf{s}_g}(t)$ being a scalar,
(e) follows from the independence of $ \mathbf{z}_{\mathbf{s}_g}(t)$ from $\mathbf{e}_{\s}(t)$,
(f) follows from Lemma~\ref{lemma:eigen_bound} (discussed in Appendix~\ref{sec:trace_ineq}) with,
\begin{align}
&\lambda_{max} \left( \mathbb{E} \left( \O^T_{\mathbf{s}_g} \mathbf{z}_{\mathbf{s}_g}(t)  \left( \O^T_{\mathbf{s}_g} \mathbf{z}_{\mathbf{s}_g}(t) \right)^T \right) \right) 
\nonumber\\
& \quad \quad = \lambda_{max} \left( \O^T_{\mathbf{s}_g} \mathbf{M}_{\mathbf{s}_g} \O_{\mathbf{s}_g} \right).
\end{align}
Finally, (g) follows from \eqref{eq:expectation_bound_general_case_prediction_fast} for large enough $N_B$.
This completes the variance analysis of $\zeta_l$, and clearly $\zeta_l 
$
has vanishingly small variance as $N_B \rightarrow \infty$.
As a consequence, the variance of
the cross term
$\displaystyle{ 2 \E_{N,t_1} \left(  \mathbf{z}^T_{\mathbf{s}_g}  \O_{\mathbf{s}_g}  \mathbf{e}_{\s} \right) 
= \frac{2}{n} \sum_{l=0}^{n - 1}  \zeta_l}$
is also vanishingly small for $N_B \rightarrow \infty$
(follows from the Cauchy-Schwarz inequality).
Since the cross term
$2 \E_{N,t_1} \left(  \mathbf{z}^T_{\mathbf{s}_g}  \O_{\mathbf{s}_g}  \mathbf{e}_{\s} \right)$
has zero mean and vanishingly small variance, by the Chebyshev inequality,
$ \left| 2 \E_{N,t_1} \left(  \mathbf{z}^T_{\mathbf{s}_g}  \O_{\mathbf{s}_g}  \mathbf{e}_{\s} \right)  \right| \leq \epsilon_1$ holds w.h.p., and this completes
the proof of \eqref{eq:3epsilon_1_bound}.
\subsection{Bounds on the trace of product of symmetric matrices} \label{sec:trace_ineq}
A useful lemma from \cite{trace_ineq_wang} can be described as follows.
\begin{lemma} \label{lemma:eigen_bound}
If $\mathbf{A}$ and $\mathbf{B}$ are two symmetric matrices in $\mathbb{R}^{n \times n }$, and $\mathbf{B}$ is positive semi-definite:
\begin{align}
\lambda_{min} \left ( \mathbf{A} \right) tr \left ( \mathbf{B} \right) \leq tr \left ( \mathbf{A} \mathbf{B} \right) \leq \lambda_{max} \left ( \mathbf{A} \right) tr \left ( \mathbf{B} \right).
\end{align}
\end{lemma}
\subsection{Results for the filtering version of the Kalman filter} \label{sec:filtering}
As stated in Section~\ref{sec:setup},
we use the prediction version of the Kalman filter for deriving our main results
in this paper.
Proving similar results for the filtering version can be done using the same techniques used for the prediction version.
In the remainder of this Section, we will first describe the
filtering setup with some additional notation,
and then describe our effective attack detector for the filtering setup.

\subsubsection{Filtering setup and additional notation} 
The state estimate update rule for the
filtering version of the Kalman filter (in steady state)
is as shown below \cite{kailath_book}:
\begin{align}
\hat{\mathbf{x}}(t) & = \hat{\mathbf{x}}^{(P)}(t) + \mathbf{L} \left ( \mathbf{y}(t) - \mathbf{C} \hat{\mathbf{x}}^{(P)}(t) \right), 
 \\
\hat{\mathbf{x}}^{(P)}(t+1) & = \mathbf{A} \hat{\mathbf{x}}(t) ,\label{eq:filtering_update}
\end{align}
where $\mathbf{L}$ is the steady state Kalman filter gain,
and $\hat{\mathbf{x}}(t)$ is the (filtered) state estimate at time $t$
(which also depends on the output at time $t$).
We denote by $\mathbf{L}_{\s}$ the steady state Kalman filter gain
when only outputs from sensor subset $\s \subseteq \{1,2,\ldots, p\}$ are used,
and use
$\mathbf{F}^\star_{\s}$ for
the corresponding filtering error covariance matrix.
In addition, we use the following notation for the senor noise in subset
$\s = \{i_1, i_2, \ldots , i_{|\s|}\}$
at time $t$:
\begin{align} 
\tilde{\mathbf{v}}_{\s}(t) = \begin{bmatrix}
v_{i_1}(t) \\ v_{i_2}(t) \\ \vdots \\ v_{i_{|\s|}}(t)
\end{bmatrix},
\label{eq:filtering_noise_notation}
\end{align}
and define $\Delta_{\s}$ as shown below:
\begin{align}
\Delta_{\s} =\E \left( \mathbf{z}_{\mathbf{s}}(t) 
\tilde{\mathbf{v}}^T_{\s}(t)  \mathbf{L}^T_{\s} 
\O^T_{\mathbf{s}} \right) , \label{eq:define_Delta}
\end{align}
where $\mathbf{z}_{\mathbf{s}}(t) = \J_{\mathbf{s}}  \bar{\mathbf{w}}(t) + \bar{\mathbf{v}}_{\mathbf{s}}(t)$ (as defined for the prediction setup in Section~\ref{sec:detector}).
Note that $\Delta_{\s}$ can be easily expressed
(after evaluating the expectation in \eqref{eq:define_Delta})
in terms of $\sigma^2_v$, $\mathbf{L}^T_{\s}$, and
$\O^T_{\mathbf{s}}$; we define
$\Delta_{\s}$ just for conveniently describing our detector
and proving its performance guarantees.
In the prediction setup, we limited the the adversary through Assumptions $1$-$5$;
however, in the filtering setup, to show similar results,
we require stronger versions of Assumptions \ref{ass:A1} and
\ref{ass:A2}
as described below:
\begin{assumption}
The adversary's knowledge at time $t$ is statistically independent of $\mathbf{w}(t')$ for $t' \geq t$, i.e., $\mathbf{a}(t)$ is statistically independent of $\{\mathbf{w}(t')\}_{t' \geq t}$.
\label{ass:A3}
\end{assumption}
\begin{assumption}
For an attack-free sensor \mbox{$i \in \{1,2,\ldots, p\} \setminus \pmb{\kappa}$},
the adversary's knowledge at time $t$ (and hence $\mathbf{a}(t)$) is statistically independent of $\{ v_i(t') \}_{t' \geq t}$.
\label{ass:A4}
\end{assumption}
Using these assumptions, we define the
effective attack detection problem for the filtering setup as follows.
\begin{definition}[\textbf{$(\epsilon,\mathbf{s})$-Effective Attack (filtering)}]
Consider the linear dynamical system under attack
$\pmb{\Sigma_a}$
as defined in~\eqref{eq:vector_system_model},
and a $k$-adversary satisfying Assumptions 1-3 and Assumptions 6-7.
For the set of sensors $\mathbf{s}$, an $\epsilon > 0$,
and a large enough $N \in \mathbb{N}$,
an attack signal is called $(\epsilon, \mathbf{s})$-effective
at
time $t_1$
if the following bound holds:
$$tr \left(\E_{N,t_1}{\left(\e_{\mathbf{s}} \e_{\mathbf{s}}^T\right)} \right) > tr( \mathbf{F}^{\star}_{\mathbf{s}})  + \epsilon,$$
where $\e_{\mathbf{s}}(t) = \x(t) - \hat{\x}_{\mathbf{s}}(t)$
(with $\hat{\x}_{\mathbf{s}}(t)$ being the filtered state estimate).\label{def:effective_attack_filtering}
\end{definition}
Note that, compared to Definition \ref{def:effective_attack}
for the prediction setup,
we have just replaced $\mathbf{P}^{\star}_{\s}$ by $\mathbf{F}^{\star}_{\mathbf{s}}$
as shown above.
In the following subsections, we describe the effective attack detector
for the filtering setup
and prove its performance guarantees.

\subsubsection{$\epsilon$-effective attack detector for filtering setup}
The effective attack detector for the filtering setup
is described in Algorithm~\ref{alg:detector_filtering}.
\begin{algorithm}[t]
\caption{\textsc{Filtering Attack-Detect}$(\mathbf{s},t_1)$}
\begin{algorithmic}[1]
\STATE  Run a Kalman filter that uses all measurements
from sensors indexed by $\mathbf{s}$
until time $t_1 $ and compute the estimate $\hat{\x}_{\mathbf{s}} (t_1) \in \R^n$.
\STATE Recursively repeat the previous step $N-1$ times to calculate all estimates $\hat{\x}_{\mathbf{s}} (t) \in \R^n$, $ \forall t  \in G = \{t_1, t_1 + 1, \ldots, t_1 + N-1 \}$.
\STATE For time $t \in G$,
calculate the \textit{block} residue:
\begin{align*}
\mathbf{r}_{\mathbf{s}}(t) =  \bar{\mathbf{y}}_{\mathbf{s}}(t) -  \O_{\mathbf{s}} \hat{\x}_{\mathbf{s}}(t)\quad  \forall t \in G.
\end{align*}
\IF{block residue test defined below holds,
\begin{align}
& \E_{N,t_1} \left(  \res_{\mathbf{s}}  \res_{\mathbf{s}}^T \right)  - 
\left( \O_{\mathbf{s}} \mathbf{F}^\star_{\s} \O^T_{\mathbf{s}} 
+  \mathbf{M}_{\s} - \Delta_{\s} - \Delta_{\s}^T 
\right) \nonumber \\
 & \qquad \qquad \qquad \qquad \qquad \qquad \quad  \preccurlyeq   \eta \; \mathbf{1}_{n|{\mathbf{s}}|} \mathbf{1}^T_{n|{\mathbf{s}}|},\label{eq:fast_block_residue_test_filtering}
\end{align}
where $0 < \eta \leq     \left( \frac {\lambda_{\min,\mathbf{s} \setminus k}} {3n(|{\mathbf{s}}|-k)} \right) \epsilon$ ,
}
\STATE assert $\hat{d}_{\text{attack},\mathbf{s}}(t_1) := 0$
\ELSE
\STATE assert $\hat{d}_{\text{attack},\mathbf{s}}(t_1):= 1$
\ENDIF
\STATE \textbf{return} $(   \hat{d}_{\text{attack},\mathbf{s}}(t_1) , \{ \hat{\x}_{\mathbf{s}}(t) \}_{t \in G} )$
\end{algorithmic}
\label{alg:detector_filtering}
\end{algorithm}
Compared to Algorithm~\ref{alg:detector} for the prediction
setup, Algorithm~\ref{alg:detector_filtering} mainly differs in the residue test;
$\mathbf{F}^\star_{\s}$ is used in place of $\mathbf{P}^\star_{\s}$, 
and the extra terms $- \Delta_{\s} - \Delta_{\s}^T$ account for the
dependence of the estimation error at time $t$ on the sensor noise at time $t$
(details in the following subsection on performance guarantees).
Note that the expected value of
$\res_{\mathbf{s}}(t)  \res_{\mathbf{s}}^T(t)$
in the absence of attacks is exactly equal to
$\O_{\mathbf{s}} \mathbf{F}^\star_{\s} \O^T_{\mathbf{s}} 
+  \mathbf{M}_{\s} - \Delta_{\s} - \Delta_{\s}^T$;
as in the prediction version,
the residue test basically checks if
the sample average
of $\res_{\mathbf{s}}(t)  \res_{\mathbf{s}}^T(t)$
in the presence of attacks is close to its
expected value in the absence of attacks.
\subsubsection{Performance guarantees for Algorithm~\ref{alg:detector_filtering}}
The following lemma states the performance guarantees for
Algorithm~\ref{alg:detector_filtering} in the context of
detecting $\epsilon$-effective attacks in the filtering
setup.

\begin{lemma} \label{lemma:detection_guarantee_filtering}
Let the linear dynamical system as defined in~\eqref{eq:vector_system_model}
be $2{k}$-sparse observable.
Consider a $k$-adversary satisfying Assumptions $1-3$ and
Assumptions $6-7$, and a sensor subset $\mathbf{s} \subseteq\{1,2,\ldots, p\}$ with $|\mathbf{s}| \geq p - k$.
For any $\epsilon > 0$ and $\delta > 0$, there exists a large enough time window length $N$ such that when Algorithm~\ref{alg:detector_filtering} terminates with $\hat{d}_{\text{attack},\mathbf{s}}(t_1) = 0 $, the following probability bound holds:
\begin{align}
 \mathbb{P} \Big(  tr\left(\E_{t_1,N} \left(\e_{\mathbf{s}} \e^T_\mathbf{s}\right) -   \mathbf{F}^\star_{\mathbf{s}}   \right) \le  \epsilon \Big) \ge 1 - \delta,
 \label{eq:estimation_property_filtering}
\end{align}
where $\e_{\mathbf{s}}(t) = \x(t) - \hat{\x}_{\mathbf{s}}(t)$
is the (filtering) state estimation error.
\label{lemma:detector_guarantee_filtering}
\end{lemma}

\begin{proof}
The proof is similar to that
for the prediction version
(Lemma~\ref{lemma:detector_guarantee}).
The main difference lies in the cross term
analysis; it is more involved than in the
prediction version due to the dependence of
estimation error at time $t$ on the sensor noise at
time $t$. We describe the proof details below.

Since we assume that the set $\mathbf{s}$ has cardinality $\vert \mathbf{s} \vert \ge p - k$,
we can conclude that
there exists a subset $\mathbf{s}_g \subset \mathbf{s}$ with cardinality  $\vert \mathbf{s}_g \vert \ge p -2k$ sensors such that all its sensors are attack-free
(subscript $g$ in $\mathbf{s}_g$ stands for \textit{good} sensors in $\mathbf{s}$). 
Hence, by decomposing the set $\mathbf{s}$  into an attack-free set $\mathbf{s}_g$ and a potentially attacked set $\s \setminus \mathbf{s}_g$, we can conclude that, after a permutation similarity transformation for \eqref{eq:fast_block_residue_test_filtering},
the following holds for the attack-free subset $\mathbf{s}_g$:
\begin{align}
 \E_{N,t_1} \left(  \res_{\mathbf{s}_g}  \res_{\mathbf{s}_g}^T \right)  -\O_{\mathbf{s}_g}  \mathbf{F}^\star_{\mathbf{s}} \O_{\mathbf{s}_g}^T   - \M_{\mathbf{s}_g} & 
+ \Delta_{\s_g} + \Delta_{\s_g}^T  \nonumber \\
& \preccurlyeq   \eta \; \mathbf{1}_{n(|\mathbf{s}|-k)} \mathbf{1}^T_{n(|\mathbf{s}|-k)},\label{eq:fast_block_residue_test_filtering_subset}
\end{align}
where $\Delta_{\s_g} = \E \left( \mathbf{z}_{\s_g}(t) 
\tilde{\mathbf{v}}^T_{\s}(t)  \mathbf{L}^T_{\s} 
\O^T_{\mathbf{s}_g} \right) $.
Therefore,
\begin{align}
& tr \left(  \E_{N,t_1} \left(  \res_{\mathbf{s}_g}  \res_{\mathbf{s}_g}^T \right)
 -\O_{\mathbf{s}_g}  \mathbf{F}^\star_{\mathbf{s}} \O_{\mathbf{s}_g}^T  
 - \M_{\mathbf{s}_g} + \Delta_{\s_g} + \Delta_{\s_g}^T  \right) \nonumber \\
& \qquad \qquad \qquad \qquad \qquad \qquad   \leq    n(|\mathbf{s}|-k)  \eta = \epsilon_1.\label{eq:fast_block_residue_test_filtering_subset_trace}
\end{align}
As in the prediction version,
we can rewrite $tr \left (\E_{N,t_1} \left( \res_{\mathbf{s}_g} \res_{\mathbf{s}_g}^T \right) \right)$
as:
\begin{align}
& tr \left(  \E_{N,t_1} \left(  \res_{\mathbf{s}_g}  \res_{\mathbf{s}_g}^T \right) \right) 
\nonumber \\
& \qquad =
tr \left(  \O_{\mathbf{s}_g} \E_{N,t_1} \left( \mathbf{e}_{\mathbf{s}} \mathbf{e}_{\mathbf{s}}^T \right) \O_{\mathbf{s}_g}^T \right)
+  tr \left ( \E_{N,t_1} \left( \z_{\mathbf{s}_g} \z_{\mathbf{s}_g}^T \right) \right) 
\nonumber  \\ & \qquad
  + 2 \E_{N,t_1} \left(  \mathbf{e}_{\mathbf{s}}^T  \O^T_{\mathbf{s}_g}  \z_{\mathbf{s}_g} \right).
\label{eq:residue_simplification_1_filtering}
\end{align}
By combining 
\eqref{eq:fast_block_residue_test_filtering_subset_trace} and \eqref{eq:residue_simplification_1_filtering}:
\begin{align}
&  tr \left ( 
\O_{\mathbf{s}_g} \E_{N,t_1} \left( \mathbf{e}_{\mathbf{s}} \mathbf{e}_{\mathbf{s}}^T \right) \O_{\mathbf{s}_g}^T  
 - \O_{\mathbf{s}_g}  \mathbf{F}^\star_{\mathbf{s}} \O_{\mathbf{s}_g}^T \right)
 \nonumber \\
& \qquad \qquad \leq 
  tr \left ( \M_{\mathbf{s}_g}  \right)   
-  tr \left( \E_{N,t_1} \left( 
\z_{\mathbf{s}_g} \z_{\mathbf{s}_g}^T \right)  \right)
 +  \epsilon_1 
\nonumber \\ & \qquad \qquad \quad 
-  2\E_{N,t_1} \left(  \mathbf{e}_{\mathbf{s}}^T  \O^T_{\mathbf{s}_g}  \z_{\mathbf{s}_g} \right)  - 2 tr \left(  \Delta_{\s_g} \right)
\nonumber \\
& \qquad \qquad \stackrel{(a)}\leq 2 \epsilon_1  -  2\E_{N,t_1} \left(  \mathbf{e}_{\mathbf{s}}^T  \O^T_{\mathbf{s}_g}  \z_{\mathbf{s}_g} \right) 
- 2 tr \left ( \Delta_{\s_g} \right) \label{eq:LLN_z_filtering}
 \\
 & \qquad \qquad \stackrel{(b)}\leq 3 \epsilon_1 , \label{eq:3epsilon_1_bound_filtering}
\end{align}
where $(a)$ follows w.h.p. due to the law of large numbers (LLN)
for large enough $N$ (as shown in Appendix \ref{sec:appendix_LLN_proof}),
and $(b)$ follows w.h.p. by showing that the cross term $2\E_{N,t_1} \left(  \mathbf{e}^T  \O^T_{\mathbf{s}_g}  \z_{\mathbf{s}_g} \right)$
has mean equal to $- 2 tr \left ( \Delta_{\s_g} \right)$
and vanishingly small variance for large enough $N$.
The cross term analysis is described in detail in Appendix \ref{sec:appendix_cross_term_analysis_filtering}.
Using \eqref{eq:3epsilon_1_bound_filtering},
the following holds:
\begin{align}
tr \left ( 
 \E_{N,t_1} \left(   \mathbf{e}_{\mathbf{s}} \mathbf{e}_{\mathbf{s}}^T -  \mathbf{F}^\star_{\mathbf{s}} \right)
\O^T_{\mathbf{s}_g}  \O_{\mathbf{s}_g}  
\right) 
&  \leq  3 \epsilon_1, 
\end{align} 
and hence, we get the following bound which completes the proof:
\begin{align}
 tr \left ( 
 \E_{N,t_1} \left( \mathbf{e}_{\mathbf{s}} \mathbf{e}_{\mathbf{s}}^T \right)  
 -  \mathbf{F}^\star_{\mathbf{s}} \right)
& \stackrel{(c)} \leq \frac{3 \epsilon_1}{\lambda_{min} \left (  \O^T_{\mathbf{s}_g}  \O_{\mathbf{s}_g} \right) }  \stackrel{(d)} \leq \frac{3 \epsilon_1}{\lambda_{\min,\mathbf{s} \setminus k}}
 \leq \epsilon \label{eq:detection_final_bound_filtering} 
\end{align} 
where $(c)$ follows from Lemma \ref{lemma:eigen_bound} in Appendix \ref{sec:trace_ineq} and $(d)$ follows from the definition of $\lambda_{\min,\mathbf{s} \setminus k}$.
Note that, it follows from $\vert \mathbf{s}_g \vert \ge p - 2k$ and $2k$-sparse observability, that both $\lambda_{min} \left (  \O^T_{\mathbf{s}_g}  \O_{\mathbf{s}_g} \right)$ and $\lambda_{\min,\mathbf{s} \setminus k}$ are bounded away from zero.
\end{proof}
Using Lemma~\ref{lemma:detector_guarantee_filtering},
deriving results for secure state estimation in the filtering
setup is straightforward,
and we skip the details for brevity.
\subsection{Cross term analysis for filtering and proof of \eqref{eq:3epsilon_1_bound_filtering}} \label{sec:appendix_cross_term_analysis_filtering}
As in the prediction setup,
the cross term $ 2 \E_{N,t_1} \left(  \mathbf{z}^T_{\mathbf{s}_g}  \O_{\mathbf{s}_g}  \mathbf{e}_{\s} \right) $ can be written down as a sum of $n$ terms as shown below:
\begin{align}
2 \E_{N,t_1} \left(  \mathbf{z}^T_{\mathbf{s}_g} 
 \O_{\mathbf{s}_g}  \mathbf{e}_{\s} \right) 
&\stackrel{(a)} = \frac{2}{n} \sum_{l=0}^{n -1 } 
\left( \frac{1}{N_B} \sum_{t \in G_l}  \mathbf{z}^T_{\mathbf{s}_g}
(t)   \O_{\mathbf{s}_g}  \mathbf{e}_{\s}(t)   \right)  \nonumber\\
&= \frac{2}{n} \sum_{l=0}^{n - 1}  \zeta_l, \nonumber
\end{align}
where (a) follows from partitioning time window $G$ (of size $N$)
into $n$ groups $G_0, G_1, \ldots G_{n-1}$ (each of size $N_B$)
such that $G_l = \{ t |  \left( \left( t-t_1 \right) \; mod \; n \right) = l \}$.
Now, we will show that each $ \zeta_l$ has mean equal to $- tr \left( \Delta_{\s_g} \right)$
and vanishingly small variance for large enough $N$.
The mean analysis can be done as shown below:
\allowdisplaybreaks[4]{
\begin{align}
\mathbb{E} \left (  \zeta_l \right) 
 & = \mathbb{E} \left (   \frac{1}{N_B} \sum_{t \in G_l} \mathbf{z}^T_{\mathbf{s}_g}(t)   \O_{\mathbf{s}_g}  \mathbf{e}_{\s}(t)  \right)  \nonumber \\
& \stackrel{(a)}
= \mathbb{E} \left (   \frac{1}{N_B} \sum_{t \in G_l} 
\mathbf{z}^T_{\mathbf{s}_g}(t)   \O_{\mathbf{s}_g}  \left(  \tilde{\mathbf{e}}_{\s}(t)  - \mathbf{L}_{\s}   \tilde{\mathbf{v}}_{\s}(t)  \right) \right)  \nonumber \\
& = \mathbb{E} \left (   \frac{1}{N_B} \sum_{t \in G_l} 
\mathbf{z}^T_{\mathbf{s}_g}(t)   \O_{\mathbf{s}_g}    \tilde{\mathbf{e}}_{\s}(t)  \right) \nonumber \\
& \quad - \mathbb{E} \left (   \frac{1}{N_B} \sum_{t \in G_l} 
\mathbf{z}^T_{\mathbf{s}_g}(t)   \O_{\mathbf{s}_g} \mathbf{L}_{\s}   \tilde{\mathbf{v}}_{\s}(t)  \right)  \nonumber \\
 & \stackrel{(b)}=  \frac{1}{N_B} \sum_{t \in G_l}  
\mathbb{E} \left ( \mathbf{z}^T_{\mathbf{s}_g}(t)   \right) 
\mathbb{E}   \left (  \O_{\mathbf{s}_g}  \tilde{\mathbf{e}}_{\s}(t)  \right) 
\nonumber\\
& \quad - \mathbb{E} \left (   \frac{1}{N_B} \sum_{t \in G_l} 
\mathbf{z}^T_{\mathbf{s}_g}(t)   \O_{\mathbf{s}_g} \mathbf{L}_{\s}   \tilde{\mathbf{v}}_{\s}(t)  \right)  \nonumber \\
& =  - \mathbb{E} \left (   \frac{1}{N_B} \sum_{t \in G_l} 
\mathbf{z}^T_{\mathbf{s}_g}(t)   \O_{\mathbf{s}_g} \mathbf{L}_{\s}   \tilde{\mathbf{v}}_{\s}(t)  \right) \nonumber\\
& =  - \mathbb{E} \left (   \frac{1}{N_B} \sum_{t \in G_l} 
tr \left ( \mathbf{z}^T_{\mathbf{s}_g}(t)   \O_{\mathbf{s}_g} \mathbf{L}_{\s}   \tilde{\mathbf{v}}_{\s}(t)  \right) \right) \nonumber \\
& =  - \mathbb{E} \left (   \frac{1}{N_B} \sum_{t \in G_l} 
tr \left (  \left (\O_{\mathbf{s}_g} \mathbf{L}_{\s}   \tilde{\mathbf{v}}_{\s}(t)\right)^T 
\mathbf{z}_{\mathbf{s}_g}(t)  \right) \right) \nonumber \\
& =  - \mathbb{E} \left (   \frac{1}{N_B} \sum_{t \in G_l} 
tr \left (   \tilde{\mathbf{v}}^T_{\s}(t) \mathbf{L}^T_{\s} \O^T_{\mathbf{s}_g} 
\mathbf{z}_{\mathbf{s}_g}(t) \right)\right) \nonumber \\
& =  - \mathbb{E} \left (   \frac{1}{N_B} \sum_{t \in G_l} 
tr \left (  \mathbf{z}_{\mathbf{s}_g}(t)   \tilde{\mathbf{v}}^T_{\s}(t) \mathbf{L}^T_{\s} \O^T_{\mathbf{s}_g} 
\right)\right) \nonumber \\
& =  -   \frac{1}{N_B} \sum_{t \in G_l} 
tr \left ( \mathbb{E} \left (  \mathbf{z}_{\mathbf{s}_g}(t)   \tilde{\mathbf{v}}^T_{\s}(t) \mathbf{L}^T_{\s} \O^T_{\mathbf{s}_g} 
\right)\right) \nonumber \\
& =  -   \frac{1}{N_B} \sum_{t \in G_l} tr \left ( \Delta_{\s_g} \right)
\nonumber \\
& = - tr \left ( \Delta_{\s_g} \right)
 \label{eq:zero_mean_vector_filtering_fast},
\end{align}}
where (a) follows from expressing
$\mathbf{e}_{\s}(t)$ as $\tilde{\mathbf{e}}_{\s}(t)  - \mathbf{L}_{\s}   \tilde{\mathbf{v}}_{\s}(t)$
(\emph{i.e.}, separating out the sensor noise at time $t$ component in $\mathbf{e}_{\s}(t)$),
and (b) follows from the independence of $\tilde{\mathbf{e}}_{\s}(t)$ from $\mathbf{z}^T_{\mathbf{s}_g}(t) $ (follows from assumptions \ref{ass:A3} and \ref{ass:A4}).
This implies that the cross term
$ 2 \E_{N,t_1} \left(  \mathbf{z}^T_{\mathbf{s}_g}  \O_{\mathbf{s}_g}  \mathbf{e}_{\s} \right) $
has mean equal to $-2 tr \left( \Delta_{\s_g} \right)$.
Also, using \eqref{eq:zero_mean_vector_filtering_fast} and
\eqref{eq:LLN_z_filtering},
{\allowdisplaybreaks[4]\begin{align}
 2 \epsilon_1 
& \geq 
 \mathbb{E} \left(  
 \E_{N,t_1}  \left ( tr  \left(  
 \O_{\mathbf{s}_g}   \left( \mathbf{e}_{\s}  \mathbf{e}_{\s}^T  - \mathbf{F}^\star_{\mathbf{s}}  \right) 
\O^T_{\mathbf{s}_g}  
\right )
\right)  
\right) 
 \nonumber \\
& =\mathbb{E} \left(  
 \E_{N,t_1}  \left ( tr  \left(   \left( \mathbf{e}_{\s}  \mathbf{e}_{\s}^T  - \mathbf{F}^\star_{\mathbf{s}} \right) 
\O^T_{\mathbf{s}_g}  \O_{\mathbf{s}_g} \right) \right ) \right) 
\nonumber\\
& \stackrel{(a)} \geq 
\lambda_{min} \left(  \O^T_{\mathbf{s}_g} \O_{\mathbf{s}_g}  \right)
\E  \left(  \E_{N,t_1} \left(     tr   \left( \mathbf{e}_{\s}  \mathbf{e}_{\s}^T  - \mathbf{F}^\star_{\mathbf{s}} \right)  \right ) \right),
\label{eq:expectation_bound_general_case_filtering_fast_0}
\end{align}}
where (a) follows from Lemma~\ref{lemma:eigen_bound} (discussed in Appendix~\ref{sec:trace_ineq}).
Using \eqref{eq:expectation_bound_general_case_filtering_fast_0},
\begin{align}
\E \left ( 
\E_{N, t_1} 
 \left(     tr   \left(  \mathbf{e}_{\s}  \mathbf{e}_{\s}^T \right)  \right)
\right)
& \leq  \frac{2 \epsilon_1}{\lambda_{min} \left(  \O^T_{\mathbf{s}_g} \O_{\mathbf{s}_g}  \right)}  
+ tr  \left( \mathbf{F}^\star_{\mathbf{s}} \right).
\label{eq:mean_error_boundedness_filtering}
\end{align}
We will use the above intermediate result in the variance
analysis done below.

\par The variance analysis for $\zeta_l$ can be done as shown below:
\allowdisplaybreaks[4]{
\begin{align}
& \E \left(  \left(  
 \frac{1}{N_B} \sum_{t \in G_l} \mathbf{z}^T_{\mathbf{s}_g}(t)   \O_{\mathbf{s}_g}  \mathbf{e}_{\s}(t) 
 \right)^2 \right) \nonumber\\
& = \E \left(  \left(  
 \frac{1}{N_B} \sum_{t \in G_l} \mathbf{e}^T_{\s}(t) \O^T_{\mathbf{s}_g}    \mathbf{z}_{\mathbf{s}_g}(t) 
 \right)^2 \right) \nonumber\\
& = \E \left(   
 \frac{1}{N^2_B} \sum_{t \in G_l} \left( \mathbf{e}^T_{\s}(t) \O^T_{\mathbf{s}_g}    \mathbf{z}_{\mathbf{s}_g}(t) \right)^2
 \right) \nonumber\\
& \quad + \E \left( \frac{2}{N^2_B} \sum_{t, t' \in G_l, t < t'} 
\mathbf{e}^T_{\s}(t) \O^T_{\mathbf{s}_g} \mathbf{z}_{\mathbf{s}_g}(t) 
\mathbf{e}^T_{\s}(t') \O^T_{\mathbf{s}_g} \mathbf{z}_{\mathbf{s}_g}(t')
\right) \nonumber\\
& = \E \left(   
 \frac{1}{N^2_B} \sum_{t \in G_l} \left( \mathbf{e}^T_{\s}(t) \O^T_{\mathbf{s}_g}    \mathbf{z}_{\mathbf{s}_g}(t) \right)^2
 \right) \nonumber\\
& \quad + \E \left( \frac{2}{N^2_B} \sum_{t, t' \in G_l, t < t'} 
\mathbf{e}^T_{\s}(t) \O^T_{\mathbf{s}_g} \mathbf{z}_{\mathbf{s}_g}(t) 
\tilde{\mathbf{e}}^T_{\s}(t') \O^T_{\mathbf{s}_g} \mathbf{z}_{\mathbf{s}_g}(t')
\right) \nonumber\\
& \quad - \E \left( \frac{2}{N^2_B} \sum_{t, t' \in G_l, t < t'} 
\mathbf{e}^T_{\s}(t) \O^T_{\mathbf{s}_g} \mathbf{z}_{\mathbf{s}_g}(t) 
\tilde{\mathbf{v}}^T_{\s}(t') \mathbf{L}^T_{\s}
 \O^T_{\mathbf{s}_g} \mathbf{z}_{\mathbf{s}_g}(t')
\right) \nonumber\\
& \stackrel{(a)}= \E \left(   
 \frac{1}{N^2_B} \sum_{t \in G_l} 
\left( \mathbf{e}^T_{\s}(t) \O^T_{\mathbf{s}_g}    \mathbf{z}_{\mathbf{s}_g}(t) \right)^2 \right) \nonumber\\
& \quad + \frac{2}{N^2_B} \sum_{t, t' \in G_l, t < t'} 
 \E \left(  \mathbf{e}^T_{\s}(t) \O^T_{\mathbf{s}_g} \mathbf{z}_{\mathbf{s}_g}(t) 
\tilde{\mathbf{e}}^T_{\s}(t') \O^T_{\mathbf{s}_g} \right)
\E \left(
\mathbf{z}_{\mathbf{s}_g}(t')
\right) \nonumber\\
& \quad - \E \left( \frac{2}{N^2_B} \sum_{t, t' \in G_l, t < t'} 
\mathbf{e}^T_{\s}(t) \O^T_{\mathbf{s}_g} \mathbf{z}_{\mathbf{s}_g}(t) 
\tilde{\mathbf{v}}^T_{\s}(t') \mathbf{L}^T_{\s}
 \O^T_{\mathbf{s}_g} \mathbf{z}_{\mathbf{s}_g}(t')
\right) \nonumber\\
& = \E \left(   
 \frac{1}{N^2_B} \sum_{t \in G_l} 
\left( \mathbf{e}^T_{\s}(t) \O^T_{\mathbf{s}_g}    \mathbf{z}_{\mathbf{s}_g}(t) \right)^2 \right)  + \mathbf{0} \nonumber\\
& \;\; \;  - \E \left( \frac{2}{N^2_B} \sum_{t, t' \in G_l, t < t'} 
\mathbf{e}^T_{\s}(t) \O^T_{\mathbf{s}_g} \mathbf{z}_{\mathbf{s}_g}(t) 
\tilde{\mathbf{v}}^T_{\s}(t') \mathbf{L}^T_{\s}
 \O^T_{\mathbf{s}_g} \mathbf{z}_{\mathbf{s}_g}(t')
\right) \nonumber\\
& = \E \left(   
 \frac{1}{N^2_B} \sum_{t \in G_l} 
\left( \mathbf{e}^T_{\s}(t) \O^T_{\mathbf{s}_g}    \mathbf{z}_{\mathbf{s}_g}(t) \right)^2 \right) \nonumber\\
& \;\;  -  
\frac{2}{N^2_B} \sum_{t, t' \in G_l, t < t'} 
\Big ( 
\E \left( \mathbf{e}^T_{\s}(t) \O^T_{\mathbf{s}_g} \mathbf{z}_{\mathbf{s}_g}(t) \right) 
  \nonumber \\ & \quad  \quad \quad  \quad \quad \quad  \quad \quad \quad \quad \quad \quad 
\times  \E \left(
\tilde{\mathbf{v}}^T_{\s}(t') \mathbf{L}^T_{\s}
 \O^T_{\mathbf{s}_g} \mathbf{z}_{\mathbf{s}_g}(t')
\right)
\Big)
  \nonumber\\
& \stackrel{(b)} = 
\E \left(   
 \frac{1}{N^2_B} \sum_{t \in G_l} 
\left( \mathbf{e}^T_{\s}(t) \O^T_{\mathbf{s}_g}    \mathbf{z}_{\mathbf{s}_g}(t) \right)^2 \right) \nonumber\\
& \quad  +  \frac{2}{N^2_B} \sum_{t, t' \in G_l, t < t'} 
\left( tr \left( \Delta_{\s_g} \right) \right)^2  
 \nonumber\\
& = \E \left(   
 \frac{1}{N^2_B} \sum_{t \in G_l}
\left( \mathbf{e}^T_{\s}(t) \O^T_{\mathbf{s}_g}    \mathbf{z}_{\mathbf{s}_g}(t) \right)^2 \right) \nonumber\\
& \quad  +  \frac{N_B \left( N_B - 1 \right)}{N^2_B} 
\left( tr \left( \Delta_{\s_g} \right) \right)^2  ,
 \label{eq:filtering_intermediate_1}
\end{align}
}
where (a) follows from independence of $\mathbf{z}_{\mathbf{s}_g}(t')$ from
$\displaystyle{\mathbf{e}^T_{\s}(t) \O^T_{\mathbf{s}_g} \mathbf{z}_{\mathbf{s}_g}(t) 
\tilde{\mathbf{e}}^T_{\s}(t') \O^T_{\mathbf{s}_g}}$
for $t < t'$,
and
(b) follows from $\displaystyle{\E \left(
\tilde{\mathbf{v}}^T_{\s}(t') \mathbf{L}^T_{\s}
 \O^T_{\mathbf{s}_g} \mathbf{z}_{\mathbf{s}_g}(t')
\right) = tr \left( \Delta_{\s_g} \right)}$ and
$\displaystyle{\E \left( \mathbf{e}^T_{\s}(t) \O^T_{\mathbf{s}_g} \mathbf{z}_{\mathbf{s}_g}(t) \right)= - tr \left( \Delta_{\s_g} \right)}$.
Now, we focus on analyzing the first term in \eqref{eq:filtering_intermediate_1} as shown below.
For any $\epsilon_2 > 0$, there exists a large enough $N_B$ such that:
\begin{align}
&\E \left(   
 \frac{1}{N^2_B} \sum_{t \in G_l}  \left( \mathbf{e}^T_{\s}(t) \O^T_{\mathbf{s}_g}    \mathbf{z}_{\mathbf{s}_g}(t) \right)^2
 \right) \nonumber\\
&=\E \left(   
 \frac{1}{N^2_B} \sum_{t \in G_l}  \left(  \left (
\tilde{\mathbf{e}}_{\s}(t) - \mathbf{L}_{\s} \tilde{\mathbf{v}}_{\s}(t)
\right)^T
 \O^T_{\mathbf{s}_g}    \mathbf{z}_{\mathbf{s}_g}(t) \right)^2
 \right) \nonumber\\
&=\E \left(   
 \frac{1}{N^2_B} \sum_{t \in G_l}  \left(  \left (
\tilde{\mathbf{e}}^T_{\s}(t) - \tilde{\mathbf{v}}^T_{\s}(t) \mathbf{L}^T_{\s} 
\right)
 \O^T_{\mathbf{s}_g}    \mathbf{z}_{\mathbf{s}_g}(t) \right)^2
 \right) \nonumber\\
&=\E \left(   
 \frac{1}{N^2_B} \sum_{t \in G_l} 
 \left(
\tilde{\mathbf{e}}^T_{\s}(t)\O^T_{\mathbf{s}_g}    \mathbf{z}_{\mathbf{s}_g}(t)
 - \tilde{\mathbf{v}}^T_{\s}(t) \mathbf{L}^T_{\s}\O^T_{\mathbf{s}_g}    \mathbf{z}_{\mathbf{s}_g}(t) 
  \right)^2
 \right) \nonumber\\
&=\E \left(   
 \frac{1}{N^2_B} \sum_{t \in G_l} 
 \left(
\tilde{\mathbf{e}}^T_{\s}(t)\O^T_{\mathbf{s}_g}    \mathbf{z}_{\mathbf{s}_g}(t)
\right)^2 \right) \nonumber\\
& \; \; \; - 
\E \left(
\frac{2}{N^2_B} \sum_{t \in G_l} 
\tilde{\mathbf{e}}^T_{\s}(t)\O^T_{\mathbf{s}_g}    \mathbf{z}_{\mathbf{s}_g}(t)
\tilde{\mathbf{v}}^T_{\s}(t) \mathbf{L}^T_{\s}\O^T_{\mathbf{s}_g}    \mathbf{z}_{\mathbf{s}_g}(t) 
 \right) \nonumber \\
& \; \; \; + 
\E \left(
\frac{1}{N^2_B} \sum_{t \in G_l} 
\left(
\tilde{\mathbf{v}}^T_{\s}(t) \mathbf{L}^T_{\s}\O^T_{\mathbf{s}_g}    \mathbf{z}_{\mathbf{s}_g}(t) 
  \right)^2
 \right) \nonumber\\
&=\E \left(   
 \frac{1}{N^2_B} \sum_{t \in G_l} 
 \left(
\tilde{\mathbf{e}}^T_{\s}(t)\O^T_{\mathbf{s}_g}    \mathbf{z}_{\mathbf{s}_g}(t)
\right)^2 \right) \nonumber\\
& \; \; \; - 
\E \left(
\frac{2}{N^2_B} \sum_{t \in G_l} 
\tilde{\mathbf{e}}^T_{\s}(t)\O^T_{\mathbf{s}_g}    \mathbf{z}_{\mathbf{s}_g}(t)
\tilde{\mathbf{v}}^T_{\s}(t) \mathbf{L}^T_{\s}\O^T_{\mathbf{s}_g}    \mathbf{z}_{\mathbf{s}_g}(t) 
 \right) \nonumber \\
& \; \; \; + 
\frac{1}{N_B} 
\E \left( \left(
\tilde{\mathbf{v}}^T_{\s}(t) \mathbf{L}^T_{\s}\O^T_{\mathbf{s}_g}    \mathbf{z}_{\mathbf{s}_g}(t) 
  \right)^2
 \right) \nonumber\\
&=\E \left(   
 \frac{1}{N^2_B} \sum_{t \in G_l} 
 \left(
\tilde{\mathbf{e}}^T_{\s}(t)\O^T_{\mathbf{s}_g}    \mathbf{z}_{\mathbf{s}_g}(t)
\right)^2 \right) \nonumber\\
& \; \; \; - 2
\frac{ \left ( 
 \sum_{t \in G_l} 
\frac{\E \left ( \tilde{\mathbf{e}}^T_{\s}(t) \right)}{N_B}
\right)}{N_B}
\E \left( 
\O^T_{\mathbf{s}_g}    \mathbf{z}_{\mathbf{s}_g}(t)
\tilde{\mathbf{v}}^T_{\s}(t) \mathbf{L}^T_{\s}\O^T_{\mathbf{s}_g}    \mathbf{z}_{\mathbf{s}_g}(t) 
 \right) \nonumber \\
& \; \; \; + 
\frac{1}{N_B} 
\E \left( \left(
\tilde{\mathbf{v}}^T_{\s}(t) \mathbf{L}^T_{\s}\O^T_{\mathbf{s}_g}    \mathbf{z}_{\mathbf{s}_g}(t) 
  \right)^2
 \right) \nonumber\\
&=\E \left(   
 \frac{1}{N^2_B} \sum_{t \in G_l} 
\tilde{\mathbf{e}}^T_{\s}(t)\O^T_{\mathbf{s}_g}    \mathbf{z}_{\mathbf{s}_g}(t)
\tilde{\mathbf{e}}^T_{\s}(t)\O^T_{\mathbf{s}_g}    \mathbf{z}_{\mathbf{s}_g}(t) 
\right) \nonumber\\
& \; \; \; - 2
\frac{ \left ( 
 \sum_{t \in G_l} 
\frac{\E \left ( \tilde{\mathbf{e}}^T_{\s}(t) \right)}{N_B}
\right)}{N_B}
\E \left( 
\O^T_{\mathbf{s}_g}    \mathbf{z}_{\mathbf{s}_g}(t)
\tilde{\mathbf{v}}^T_{\s}(t) \mathbf{L}^T_{\s}\O^T_{\mathbf{s}_g}    \mathbf{z}_{\mathbf{s}_g}(t) 
 \right) \nonumber \\
& \; \; \; + 
\frac{1}{N_B} 
\E \left( \left(
\tilde{\mathbf{v}}^T_{\s}(t) \mathbf{L}^T_{\s}\O^T_{\mathbf{s}_g}    \mathbf{z}_{\mathbf{s}_g}(t) 
  \right)^2
 \right) \nonumber\\
&=\E \left(   
 \frac{1}{N^2_B} \sum_{t \in G_l} 
\tilde{\mathbf{e}}^T_{\s}(t)\O^T_{\mathbf{s}_g} \mathbf{z}_{\mathbf{s}_g}(t)
\mathbf{z}^T_{\mathbf{s}_g}(t) \O_{\mathbf{s}_g} \tilde{\mathbf{e}}_{\s}(t)  
\right) \nonumber\\
& \; \; \; - 2
\frac{ \left ( 
 \sum_{t \in G_l} 
\frac{\E \left ( \tilde{\mathbf{e}}^T_{\s}(t) \right)}{N_B}
\right)}{N_B}
\E \left( 
\O^T_{\mathbf{s}_g}    \mathbf{z}_{\mathbf{s}_g}(t)
\tilde{\mathbf{v}}^T_{\s}(t) \mathbf{L}^T_{\s}\O^T_{\mathbf{s}_g}    \mathbf{z}_{\mathbf{s}_g}(t) 
 \right) \nonumber \\
& \; \; \; + 
\frac{1}{N_B} 
\E \left( \left(
\tilde{\mathbf{v}}^T_{\s}(t) \mathbf{L}^T_{\s}\O^T_{\mathbf{s}_g}    \mathbf{z}_{\mathbf{s}_g}(t) 
  \right)^2
 \right) \nonumber\\
&=\E \left(   
 \frac{1}{N^2_B} \sum_{t \in G_l} 
tr \left( \tilde{\mathbf{e}}^T_{\s}(t)\O^T_{\mathbf{s}_g} \mathbf{z}_{\mathbf{s}_g}(t)
\mathbf{z}^T_{\mathbf{s}_g}(t) \O_{\mathbf{s}_g} \tilde{\mathbf{e}}_{\s}(t)  
\right) \right) \nonumber\\
& \; \; \; - 2
\frac{ \left ( 
 \sum_{t \in G_l} 
\frac{\E \left ( \tilde{\mathbf{e}}^T_{\s}(t) \right)}{N_B}
\right)}{N_B}
\E \left( 
\O^T_{\mathbf{s}_g}    \mathbf{z}_{\mathbf{s}_g}(t)
\tilde{\mathbf{v}}^T_{\s}(t) \mathbf{L}^T_{\s}\O^T_{\mathbf{s}_g}    \mathbf{z}_{\mathbf{s}_g}(t) 
 \right) \nonumber \\
& \; \; \; + 
\frac{1}{N_B} 
\E \left( \left(
\tilde{\mathbf{v}}^T_{\s}(t) \mathbf{L}^T_{\s}\O^T_{\mathbf{s}_g}    \mathbf{z}_{\mathbf{s}_g}(t) 
  \right)^2
 \right) \nonumber\\
&=
\E 
\left(   
 \frac{1}{N^2_B} \sum_{t \in G_l} 
tr 
\left( 
\tilde{\mathbf{e}}_{\s}(t) \tilde{\mathbf{e}}^T_{\s}(t)
\O^T_{\mathbf{s}_g} 
\mathbf{z}_{\mathbf{s}_g}(t) \mathbf{z}^T_{\mathbf{s}_g}(t) 
\O_{\mathbf{s}_g} 
\right )
\right )
 \nonumber\\
& \; \; \; - 2
\frac{ \left ( 
 \sum_{t \in G_l} 
\frac{\E \left ( \tilde{\mathbf{e}}^T_{\s}(t) \right)}{N_B}
\right)}{N_B}
\E \left( 
\O^T_{\mathbf{s}_g}    \mathbf{z}_{\mathbf{s}_g}(t)
\tilde{\mathbf{v}}^T_{\s}(t) \mathbf{L}^T_{\s}\O^T_{\mathbf{s}_g}    \mathbf{z}_{\mathbf{s}_g}(t) 
 \right) \nonumber \\
& \; \; \; + 
\frac{1}{N_B} 
\E \left( \left(
\tilde{\mathbf{v}}^T_{\s}(t) \mathbf{L}^T_{\s}\O^T_{\mathbf{s}_g}    \mathbf{z}_{\mathbf{s}_g}(t) 
  \right)^2
 \right) \nonumber\\
&= \frac{1}{N^2_B} \sum_{t \in G_l} 
tr \left( 
\E \left( 
\tilde{\mathbf{e}}_{\s}(t) \tilde{\mathbf{e}}^T_{\s}(t) 
\right)
\E \left(
\O^T_{\mathbf{s}_g} 
\mathbf{z}_{\mathbf{s}_g}(t) \mathbf{z}^T_{\mathbf{s}_g}(t) 
\O_{\mathbf{s}_g} 
\right )
\right )
 \nonumber\\
& \; \; \; - 2
\frac{ \left ( 
 \sum_{t \in G_l} 
\frac{\E \left ( \tilde{\mathbf{e}}^T_{\s}(t) \right)}{N_B}
\right)}{N_B}
\E \left( 
\O^T_{\mathbf{s}_g}    \mathbf{z}_{\mathbf{s}_g}(t)
\tilde{\mathbf{v}}^T_{\s}(t) \mathbf{L}^T_{\s}\O^T_{\mathbf{s}_g}    \mathbf{z}_{\mathbf{s}_g}(t) 
 \right) \nonumber \\
& \; \; \; + 
\frac{1}{N_B} 
\E \left( \left(
\tilde{\mathbf{v}}^T_{\s}(t) \mathbf{L}^T_{\s}\O^T_{\mathbf{s}_g}    \mathbf{z}_{\mathbf{s}_g}(t) 
  \right)^2
 \right) \nonumber\\
&= \frac{1}{N^2_B} \sum_{t \in G_l} 
tr \left( 
\E \left( 
\tilde{\mathbf{e}}_{\s}(t) \tilde{\mathbf{e}}^T_{\s}(t) 
\right)
 \left(
\O^T_{\mathbf{s}_g} 
\mathbf{M}_{\mathbf{s}_g} 
\O_{\mathbf{s}_g} 
\right )
\right )
 \nonumber\\
& \; \; \; - 2
\frac{ \left ( 
 \sum_{t \in G_l} 
\frac{\E \left ( \tilde{\mathbf{e}}^T_{\s}(t) \right)}{N_B}
\right)}{N_B}
\E \left( 
\O^T_{\mathbf{s}_g}    \mathbf{z}_{\mathbf{s}_g}(t)
\tilde{\mathbf{v}}^T_{\s}(t) \mathbf{L}^T_{\s}\O^T_{\mathbf{s}_g}    \mathbf{z}_{\mathbf{s}_g}(t) 
 \right) \nonumber \\
& \; \; \; + 
\frac{1}{N_B} 
\E \left( \left(
\tilde{\mathbf{v}}^T_{\s}(t) \mathbf{L}^T_{\s}\O^T_{\mathbf{s}_g}    \mathbf{z}_{\mathbf{s}_g}(t) 
  \right)^2
 \right) \nonumber\\
& \stackrel{(a)}\leq  
\frac{\lambda_{max} \left(
\O^T_{\mathbf{s}_g} 
\mathbf{M}_{\mathbf{s}_g} 
\O_{\mathbf{s}_g} 
\right )}{N^2_B} \sum_{t \in G_l} 
tr \left( 
\E \left( 
\tilde{\mathbf{e}}_{\s}(t) \tilde{\mathbf{e}}^T_{\s}(t) 
\right)
\right )
 \nonumber\\
& \; \; \; - 2
\frac{ \left ( 
 \sum_{t \in G_l} 
\frac{\E \left ( \tilde{\mathbf{e}}^T_{\s}(t) \right)}{N_B}
\right)}{N_B}
\E \left( 
\O^T_{\mathbf{s}_g}    \mathbf{z}_{\mathbf{s}_g}(t)
\tilde{\mathbf{v}}^T_{\s}(t) \mathbf{L}^T_{\s}\O^T_{\mathbf{s}_g}    \mathbf{z}_{\mathbf{s}_g}(t) 
 \right) \nonumber \\
& \; \; \; + 
\frac{1}{N_B} 
\E \left( \left(
\tilde{\mathbf{v}}^T_{\s}(t) \mathbf{L}^T_{\s}\O^T_{\mathbf{s}_g}    \mathbf{z}_{\mathbf{s}_g}(t) 
  \right)^2
 \right) \nonumber\\
&\stackrel{(b)}
\leq \epsilon_2, \label{eq:filtering_intermediate_3}
\end{align}
where (a) follows from
Lemma~\ref{lemma:eigen_bound} (discussed in Appendix~\ref{sec:trace_ineq}),
and (b) follows
for large enough $N_B$
from the boundedness of
$ \E \left ( 
\E_{N, t_1} 
 \left(     tr   \left(  \mathbf{e}_{\s}  \mathbf{e}_{\s}^T \right)  \right)
\right)$
as shown in \eqref{eq:mean_error_boundedness_filtering}.

\par Using \eqref{eq:filtering_intermediate_3} and
\eqref{eq:filtering_intermediate_1}, for any $\epsilon_3 > 0$, there
exists a large enough $N_B$ such that:
\begin{align}
Var(\zeta_l) & = \E(\zeta^2_l) - \left ( \E(\zeta_l) \right)^2 \nonumber\\
& \stackrel{(a)}\leq \epsilon_3 + \left( tr\left ( \Delta_{\s_g} \right) \right)^2
- \left( tr\left ( \Delta_{\s_g} \right) \right)^2  
& = \epsilon_3,
\end{align}
where (a) follows from \eqref{eq:filtering_intermediate_3} and
\eqref{eq:filtering_intermediate_1}.
This completes the variance analysis of $\zeta_l$, and clearly $\zeta_l$ 
has vanishingly small variance as $N_B \rightarrow \infty$.
As a consequence, the variance of
the cross term
$\displaystyle{ 2 \E_{N,t_1} \left(  \mathbf{z}^T_{\mathbf{s}_g}  \O_{\mathbf{s}_g}  \mathbf{e}_{\s} \right) 
= \frac{2}{n} \sum_{l=0}^{n - 1}  \zeta_l}$
is also vanishingly small for $N_B \rightarrow \infty$
(follows from the Cauchy-Schwarz inequality).
This completes the proof of \eqref{eq:3epsilon_1_bound_filtering}.

\end{document}